\documentclass[journal]{IEEEtran}
\usepackage{amsmath,amsfonts}
\usepackage{algorithmic}
\usepackage{array}
\usepackage{textcomp}
\usepackage{stfloats}
\usepackage{url}
\usepackage{verbatim}
\usepackage{graphicx}
\def\BibTeX{{\rm B\kern-.05em{\sc i\kern-.025em b}\kern-.08em
    T\kern-.1667em\lower.7ex\hbox{E}\kern-.125emX}}
\usepackage{balance}
\usepackage{verbatim}
\usepackage{amsfonts}
\usepackage{amssymb}
\usepackage{stfloats}
\usepackage{cite}
\usepackage{soul}
\usepackage{setspace}
\usepackage{graphicx}
\usepackage{psfrag}
\usepackage{amsmath}
\usepackage{array}
\usepackage{epstopdf}
\usepackage{authblk}
\usepackage{graphicx} 
\usepackage{amsthm} 
\usepackage{lipsum}
\usepackage{verbatim} 
\usepackage{authblk}
\usepackage{mathtools}
\usepackage{cuted}
\usepackage[lined,boxed,ruled]{algorithm2e}
\usepackage{booktabs}
\usepackage{subfigure}
\usepackage[font=small]{caption} 

\usepackage{ifthen}

\usepackage[usenames]{color}

\newtheorem{Definition}{Definition}

\newtheorem{Lemma}{Lemma}

\newtheorem{Proposition}{Proposition}
\newtheorem{Remark}{Remark}

\DeclareMathOperator*{\argmax}{argmax}
\DeclareMathOperator*{\argmin}{argmin}
\DeclareMathOperator{\diag}{diag}

\newcommand{\zw}[1]{\textcolor{blue}{#1}}
\newboolean{showaircomp}
\setboolean{showaircomp}{False} 

\begin{document}

\title{AirBreath Sensing: Protecting Over-the-Air Distributed Sensing Against Interference}
\author{{Zhanwei~Wang},~{Mingyao~Cui},~{Huiling~Yang},~{Qunsong~Zeng},
{Min~Sheng},~\IEEEmembership{Fellow,~IEEE},~and~{Kaibin~Huang},~\IEEEmembership{Fellow,~IEEE}
\thanks{Z.~Wang, M.~Cui, H.~Yang, Q.~Zeng, and K. Huang are with the Department of Electrical and Electronic Engineering, The University of Hong Kong, Hong Kong SAR, China (Email: \{zhanweiw, mycui, hlyang,  qszeng, huangkb\}@eee.hku.hk).
M. Sheng is with the State Key Laboratory of Integrated Service Networks, Institute of Information Science, Xidian
University, Xi’an, Shaanxi, China. (Email:
msheng@mail.xidian.edu.cn).
 Corresponding authors: Q. Zeng; K. Huang.
}
  }


\maketitle

\vspace{-1cm}

\begin{abstract}
A distinctive function of sixth-generation (6G) networks is the integration of distributed sensing and edge artificial intelligence (AI) to enable intelligent perception of the physical world. 
This resultant platform, termed integrated sensing and edge AI (ISEA), is envisioned to enable a broad spectrum of Internet-of-Things (IoT) applications, including remote surgery, autonomous driving, and holographic telepresence. 
Recently, the communication bottleneck confronting the implementation of an ISEA system is overcome by the development of over-the-air computing (AirComp) techniques, which facilitate simultaneous access through over-the-air data feature fusion. 
Despite its advantages, AirComp with uncoded transmission remains vulnerable to interference.
To tackle this challenge, we propose \emph{AirBreath sensing}, a spectrum-efficient framework that cascades feature compression and spread spectrum to mitigate interference without bandwidth expansion. 
This work reveals a fundamental tradeoff between these two operations under a fixed bandwidth constraint: increasing the compression ratio may reduce sensing accuracy but allows for more aggressive interference suppression via spread spectrum, and vice versa. 
This tradeoff is regulated by a key variable called \emph{breathing depth}, defined as the feature subspace dimension that matches the processing gain in spread spectrum.
To optimally control the breathing depth,  we mathematically characterize and optimize this aforementioned tradeoff by designing a tractable surrogate for sensing accuracy, measured by classification discriminant gain (DG). 
Experimental results on real datasets demonstrate that AirBreath sensing effectively mitigates interference in ISEA systems, and the proposed control algorithm achieves near-optimal performance as benchmarked with a brute-force search.

 \end{abstract}

\begin{IEEEkeywords}
Integrated sensing and edge AI, spread spectrum, over-the-air computation, and interference suppression.
\end{IEEEkeywords}

%
\section{Introduction}
\label{Sec:intro}

The \emph{sixth-generation} (6G) networks feature two new functions. One is distributed sensing that enables environmental perception through cooperation between many on-device sensors~\cite{Saad_6G,CHEN2025,Huawei2022}. 
The other is edge \emph{artificial intelligence} (AI) that is envisioned to support real-time inference and decision-making via deploying distributed AI and machine learning algorithms at the network edge~\cite{Shi2020CommEffEdgeAI,ZW2024ultra-LoLa,yang2025optBS}. 
The natural convergence of distributed sensing and edge AI, termed \emph{integrated sensing and edge AI} (ISEA), is poised to transform a wide range of Internet-of-Things applications, for example, remote surgery, autonomous driving, and holographic telepresence~\cite{liu2025ISEAearly,cui2025quansensing,wei20223u,wei2023differential}. 
Nevertheless, the implementation of ISEA in practical mobile networks faces two challenges. One is a communication bottleneck resulting from the transmission of high-dimensional features from many sensors to an edge server for aggregation and inference. 
A popular solution is to realize over-the-air feature aggregation using a technique called \emph{over-the-air computing} (AirComp)~\cite{GX-BBA}. The other challenge is the vulnerability of transmission to interference from coexisting services, neighbouring cells, and attacks. To tackle the challenge, we consider an over-the-air distributed sensing system, which employs AirComp, and propose the framework of AirBreath sensing, which provides a spectrum-efficient way to suppress interference while maintaining satisfactory sensing performance.

The basic operation of an ISEA system is to upload and aggregate sensor observations as input into an inference model at the server. 
To be specific, each sensor extracts features from observations (e.g., images) using a lightweight, pre-trained model and transmits these features to the edge server for aggregation and inference. 
In the area of ISEA (also called edge-inference), research focuses on overcoming communication bottlenecks arising from uploading high-dimensional features. 
The first approach focuses on overcoming the communication bottleneck through overhead-reduction techniques, including feature compression~\cite{liu2025semantic,zw2025AIoutage,huang2025d}, quantization~\cite{zeng2025ultra-lola}, uncoded transmission~\cite{zeng2024knowledge}, and access control strategies~\cite{who2com,cang2024joint}. 
The second approach pertains to the split-inference architecture, where the global AI model is divided into on-device and server components to optimize overall system performance. Relevant research exploits an inherent degree-of-freedom, namely splitting point, to enhance the \emph{end-to-end} (E2E) inference performance under constraints like bandwidth and latency~\cite{Chen-TWC-2019,ZJ-CoM-2020}. 
When there are many sensors, the preceding two approaches are challenged by the scalability issue as the required radio resources for multi-access (or otherwise latency) increases linearly with the number of sensors. 
The issue can be addressed by the third approach, i.e., the aforementioned  AirComp, referring to a class of task-oriented multi-access schemes aimed at mitigating the communication bottleneck. 
Unlike traditional orthogonal access methods, AirComp enables simultaneous transmission and over-the-air aggregation of features by leveraging waveform superposition (see the survey~\cite{OTA_survey}).
Researchers have applied AirComp to realize over-the-air max-pooling for distributed multi-view sensing~\cite{Zhiyan-AirPooling} and over-the-air voxel fusion for distributed point-cloud sensing~\cite{liu2025over}. On the other hand, a unique tradeoff in over-the-air sensing with AirComp that varies the number of active sensors induces a tradeoff between channel noise and sensing accuracy due to the aggregation gain and the number of views, respectively. Efforts have been made to characterize and optimize this tradeoff through designing sensor-selection strategies~\cite{Xu-JSAC}.

For its efficiency, AirComp is also integrated into the proposed AirBreath framework. 
However, a key limitation of AirComp is its reliance on uncoded linear analog modulation, which exposes signals to interference and noise. 
Extensive research has been devoted to addressing this issue by exploiting cooperation among users or across cells. 
One strategy involves the server controlling the transmission power of devices according to the criterion of minimum AirComp error, which is measured by the discrepancy between the channel-distorted aggregated features of sensing data and their ground truths~\cite{cao2020cooperative,cao2020optimized,zhang2022interference}. 
In multi-cell systems, classic interference alignment algorithms have been further developed for AirComp to maximize its \emph{degrees of freedom} (DoF) by partitioning the signal space into two orthogonal subspaces: one for AirComp and the other for containing interference~\cite{qiao-twocell}. 
While earlier works primarily used AI-agnostic metrics, recent studies have focused on intelligent sensing, estimation, and actuation (ISEA) scenarios, such as human motion recognition tasks. 
These studies aim to mitigate channel noise in AirComp systems by considering E2E performance metrics related to sensing accuracy, such as average discriminant gain, and by developing advanced AirComp schemes through the joint optimization of sensing power, feature precoding, and receive beamforming~\cite{wen2023task,wen2023taskOTA,huang2025visual}. 
Despite these advances, existing approaches generally require accurate \emph{channel state information} (CSI) on interference channels. 
This is impractical in the absence of multi-cell cooperation or when interference originates from external networks or jammers.

Recently, a new technique called \emph{spectrum breathing} has been proposed to address this issue in over-the-air \emph{federated learning} (FL) systems, enabling interference mitigation without requiring CSI of the interference link~\cite{ZW_spectrum}. 
This approach involves cascading gradient-pruning and spread-spectrum operations at the server during each FL communication round. 
Under a spectrum constraint, gradient pruning is first applied to reduce the data bandwidth to create room for subsequent spectrum spreading, which facilitates interference suppression through spectrum despreading at the receiver. 
The iterative process of spectrum contraction (via pruning) and spreading during learning mimics the rhythm of human breathing, inspiring the technique’s name.
In FL with spectrum breathing, a fundamental tradeoff emerges: gradient pruning slows down convergence, while spread spectrum enhances interference suppression. 
The tradeoff is regulated by a variable called \emph{breathing depth}, defined as the pruning ratio (or, equivalently, the processing gain for spread spectrum).
The design of its optimal control strategy by balancing the preceding tradeoff is crucial for maximizing the spectrum breathing's effectiveness. 
This requires tractable analysis of FL convergence speed as a function of breathing depth.

For its practicality, the spectrum-breathing technique is adopted in the current work to develop a strategy for enhancing the robustness of ISEA systems against interference, leading to the proposed framework of \emph{AirBreath sensing}. 
Although both FL and ISEA systems employ over-the-air aggregation, the optimal control of breathing depth fundamentally differs between the two due to distinct algorithms and performance metrics. 
Specifically, an ISEA system focuses on inference and assumes a pre-trained AI model without involving learning as in its FL counterpart; 
the E2E performance metric for ISEA systems is sensing accuracy, in contrast to the convergence speed emphasized in FL.
As a result, the optimal control of breathing depth in ISEA systems gives rise to two new challenges. 
The first is to mathematically derive the tradeoff between accuracy degradation caused by feature pruning and accuracy enhancement enabled by interference suppression via spread spectrum. 
The second is to devise an optimal control strategy for breathing depth that effectively balances this tradeoff to maximize E2E sensing accuracy.
The design of the AirBreath framework aims to address these challenges. 
To address these challenges, we propose an ISEA-centric spectrum breathing framework for over-the-air distributed sensing, termed AirBreath sensing. Its novelty lies in the cascaded integration of feature compression and spread spectrum under bandwidth constraints. The fundamental difference of AirBreath sensing from the FL counterpart lies in the approach to bandwidth contraction. Specifically, AirBreath sensing reduces feature dimensionality through a low-rank compression matrix, enabling importance-aware feature compression instead of random element pruning in~\cite{ZW_spectrum}. 
The major contributions and key results are highlighted as follows.
\begin{itemize}
    \item \textbf{Derivation of the fundamental tradeoff in AirBreath sensing:}
    We analyze the tradeoff between the effects of feature compression and spread spectrum on sensing accuracy, referred to as the \emph{compression-spreading tradeoff}. 
    For analytical tractability, we assume a widely used \emph{Gaussian mixture model} (GMM) for the feature distribution and employ linear classification, both of which will be generalized in the second part of this work. 
    The receive \emph{discriminant gain} (DG), a key indicator of sensing accuracy, is shown to increase monotonically with both the dimensionality of the pruned feature subspace and the processing gain. 
    Then when a bandwidth constraint is imposed, the compression-spreading tradeoff arises: increasing the transmitted feature dimensionality improves the receive DG, but the resulting decrease in processing gain makes feature transmission more susceptible to interference, thereby reducing the receive DG. 
    This tradeoff necessitates the optimization of the feature dimensions, termed the breathing depth of the considered ISEA system, to maximize sensing performance.

    \item \textbf{Optimal control of AirBreath sensing:}
    The optimal control strategy for breathing depth is developed by balancing the aforementioned tradeoff. 
    To this end, we derive a closed-form expression for the receive DG under importance-aware feature pruning. 
    For tractable design, this DG is accurately approximated by a continuous surrogate function, which is proven to be unimodal and to possess a unique maximum. 
    Using the surrogate as the optimization objective, the resulting optimal breathing depth is derived and shown to increase monotonically with both the \emph{signal-to-interference ratio} (SIR) and the number of active sensors participating in feature aggregation. 
    Furthermore, the design is extended to the more complex case of  \emph{convolutional neural network} (CNN) classification with a general feature distribution. 
    In this case, the DG surrogate is constructed as the combined sum of the effects of feature compression and spread spectrum.
    This enables the derivation of a semi-analytic optimal breathing depth, where a parameter is fitted using a training dataset.

    \item \textbf{Experimental results:}  
    The compression-spreading tradeoff and proposed surrogates are validated using synthetic and real-world datasets (e.g., GMM and ModelNet~\cite{ModelNet-Ref}, respectively). The proposed AirBreath sensing technique closely approximates the optimal performance of a brute-force search and outperforms methods with alternative interference suppression techniques.

\end{itemize}

The organization of the paper is outlined below. Section \ref{sec:ModelandMetrics} introduces the system model and defines the performance metric. In Section \ref{sec:transceiver}, we present the transceiver design of AirBreath sensing. 
A tradeoff in AirBreath sensing is uncovered in Section \ref{sec:funde-tradeoff}.
Then the found tradeoff is optimized for linear and CNN classification in Sections \ref{sec:optimization} and \ref{sec:ext_CNN}, respectively. Section \ref{sec:experiments} reports experimental results, while concluding remarks are provided in Section \ref{sec:conclusion}.

%

\section{System Models and Metrics}
\label{sec:ModelandMetrics}

\begin{figure}[t!]
\centering
\includegraphics[width=1\columnwidth]{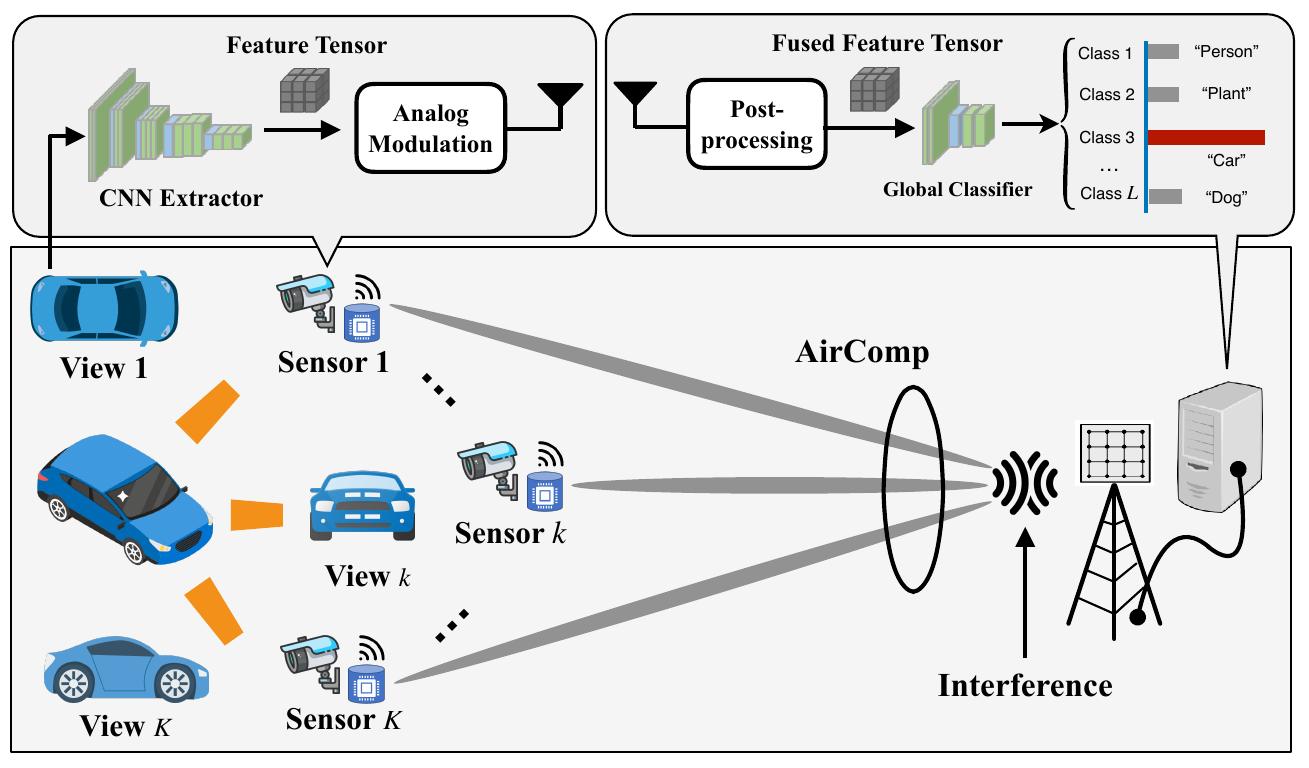}
\caption{AirComp-based ISEA system with mult-view sensing.}
\label{fig:system_diag}
\vspace{-5mm} 
\end{figure}

Consider the ISEA system illustrated in Fig.~\ref{fig:system_diag}, where an edge server performs remote object classification using AI-enabled multi-view sensing from 
$K$ distributed sensors.
Time is organized into recurring rounds, each initiated by the edge server to identify a sequence of objects.
In every round, the edge server requests the distributed sensors to upload feature vectors extracted from their captured observations.
These uploaded feature vectors are then jointly processed to perform collaborative inference and recognize the corresponding object.


\vspace{-2mm}
\subsection{Distributed Sensing Model}

\subsubsection{Local Data Distribution}
Sensor $k$  acquires the observation of the object and then extracts the feature vector, denoted as $\mathbf{x}_k\in \mathbb{R}^{D}$, using a local backbone model. 
While CNN classification is designed for generic distributions, as for linear classification, we assume the subsequent data distribution for tractability.
The feature vectors are sampled from a GMM~\cite{figueroa2019semi, ZW2024ultra-LoLa}, where each $\mathbf{x}_k$ is assigned to one of $L$ classes with equal prior probability. 
For class $\ell$, the associated feature vector subjects as a Gaussian with mean vector $\boldsymbol{\mu}_\ell \in \mathbb{R}^{D}$ and a common covariance matrix $\mathbf{C} \in \mathbb{R}^{D \times D}$. 
The mean vectors, serving as class centroids, vary across classes, whereas the covariance matrix is assumed to be identical for all classes~\cite{ZW2024ultra-LoLa,zeng2024knowledge}. 
Thereafter, the joint distribution of the $K$ feature vectors is given by:
\begin{equation} \label{eq:data_dis} (\mathbf{x}_1, \ldots, \mathbf{x}_K) \sim \frac{1}{L} \sum_{\ell=1}^L \prod_{k=1}^K \mathcal{N}\left(\mathbf{x}_k \mid \boldsymbol{\mu}_\ell, \mathbf{C}\right). \end{equation}
where  $\mathcal{N}\left(\mathbf{x}_k \mid \boldsymbol{\mu}_\ell, \mathbf{C}\right)$ represents the Gaussian \emph{probability density function} (PDF).

\subsubsection{Global Classification}
Subsequently, the extracted feature vectors ${\mathbf{x}_k}$ are sent by the distributed sensors to the server for object recognition.
They are first aggregated into a single feature vector $\overline{\mathbf{x}}$ using the commonly adopted average operation, i.e., 
\[
\overline{\mathbf{x}}=\frac{1}{K}\sum_{k=1}^K\mathbf{x}_k.
\]
The aggregated vector is then passed to the server-side classifier for inference. 
Two types of classifiers are considered.
\begin{itemize}
    \item \textit{Linear Classification:}  
    For analytical tractability, we adopt a \emph{maximum likelihood} (ML) classifier for the distribution in~\eqref{eq:data_dis}, in which the decision boundary between any two classes is a hyperplane in the feature space.  
    With uniform class priors, the ML classifier is equivalent to a \emph{maximum a posteriori} (MAP) classifier. The estimated label $\hat{\ell}$ is obtained as
    \begin{equation}
    \label{Maha_min_classifier}
        \hat{\ell} = \argmax_{\ell} \log \Pr(\overline{\mathbf{x}}|\ell)
        = \argmin_{\ell} z_\ell(\overline{\mathbf{x}}),
    \end{equation}
    where $z_\ell(\overline{\mathbf{x}})$ denotes the squared Mahalanobis distance between $\overline{\mathbf{x}}$ and the class centroid $\boldsymbol{\mu}_\ell$:
    \begin{equation}
    z_\ell(\overline{\mathbf{x}}) = (\overline{\mathbf{x}} - \boldsymbol{\mu}_\ell)^{\mathsf{T}} \mathbf{C}^{-1} (\overline{\mathbf{x}} - \boldsymbol{\mu}_\ell).
    \end{equation}


\item \textit{CNN Classification:}  
We consider a more sophisticated, analytically intractable case where feature vectors are extracted from observations using a CNN. The CNN is split into a sensor-side sub-model and a server-side sub-model, denoted by $f_{\sf sen}(\cdot)$ and $f_{\sf ser}(\cdot)$, respectively~\cite{wei2025pipelining}. The identical sensor-side model $f_{\sf sen}(\cdot)$ generates the local feature vector $\mathbf{x}_k$ from the $k$-th observation. After view aggregation, the edge server applies $f_{\sf ser}(\cdot)$ to the aggregated feature vector $\overline{\mathbf{x}}$ to produce confidence scores $\{c_1,\dots,c_L\} = f_{\sf ser}(\overline{\mathbf{x}})$, where $c_\ell$ is the confidence score for class $\ell$. The predicted label is $\hat{\ell} = \argmax_{\ell} c_\ell$.

\end{itemize}

\vspace{-2mm}

\subsection{Communication Model}

As illustrated in Fig.~\ref{fig:system_diag}, feature uploading via AirComp is affected by interference. 
To mitigate this effect, each transmitted signal undergoes \emph{feature compression} followed by \emph{spread spectrum}, whose details are provided in Sec.~\ref{sec:transceiver}. 
For clarity, we introduce the following notation. 
The $s$-th element of the compressed feature vector from sensor $k$ is scrambled by spread spectrum into a sequence $\tilde{\mathbf{x}}_{k,s}$, where each entry is referred to as a chip, and the $m$-th chip is denoted by $\tilde{X}_{k,s,m}$.

Using the notations above, AirComp realizes chip-level feature fusion and is modeled as follows.
In particular, each sensor linearly modulates spreading chips and uploads them to the edge server with a scheduled timing advance~\cite{ZW_spectrum}.
This ensures the waveforms are superposed at the edge server and realize the feature fusion naturally~\cite{GX-BBA}.
The simultaneous transmission of $(s,m)$-th chip, i.e., $\tilde{X}_{k,s,m}$, enables AirComp to yield received chip symbol, given by
\begin{equation}
\label{eq:rece_signal}
    {Y}_{s,m}= \sum_{k=1}^K p_k h_{k} \tilde{{X}}_{k,s,m} +{Z}_{s,m},
\end{equation}
where $p_{k}$  and $h_{k}\sim \mathcal{CN}(0,1)$ represent the precoding coefficient of sensor $k$ and associated channel gain with Rayleigh distribution, respectively.
${Z}_{s,m}\sim \mathcal{CN}(0,P_I)$ denotes the channel interference subjecting to Gaussian distribution.
We model the worst-case interference as Gaussian over the chip duration~\cite{Worst-Gaussian,ZW_spectrum} and neglect channel noise in the interference-limited system.

\vspace{-3mm}

\subsection{Performance Metric}

We consider the following two metrics that measure the performance of sensing tasks.

\subsubsection{Sensing Accuracy}

Consider a multi-view sensing task in which decision-making is based on object classification using well-trained AI models.
The system performance, quantified as sensing accuracy, is defined by the probability of correct classification, following the principles of statistical learning theory~\cite{Friedman-Book-2009}, and is computed as
\begin{equation}
\label{eq:sensing_acc}
    A=\frac{1}{L}\sum_{\ell=1}^L\Pr(\hat{\ell}=\ell\mid \ell),
\end{equation}
where $\Pr(\hat{\ell}=\ell\mid \ell)$ is the probability of correct label prediction  with $\ell$ and  $\hat{\ell}$  being the   ground-truth and inferred label, respectively.

\subsubsection{Discriminant Gain}
The \emph{discriminant gain} (DG)  is a popular metric to quantify the discernibility between two classes within a feature subspace.
Given the fused feature vector $\overline{\mathbf{x}}$, the DG between  class $\ell$ and $\ell'$, denoted as  $ \mathcal{G}_{\ell,{\ell'}}$, is defined as the symmetric Kullback-Leibler divergence~\cite{ZW2024ultra-LoLa}:
\begin{equation}
\begin{split}
    \mathcal{G}_{\ell,{\ell'}} = & \sf{KL}(\mathcal{N}(\boldsymbol{\mu}_\ell,\overline{\mathbf{C}})\,||\, \mathcal{N}(\boldsymbol{\mu}_{\ell'},\overline{\mathbf{C}}))\\
    &+\sf{KL}(\mathcal{N}(\boldsymbol{\mu}_{\ell'},\overline{\mathbf{C}})\,||\,\mathcal{N}(\boldsymbol{\mu}_\ell,\overline{\mathbf{C}}))\\
     =&(\boldsymbol{\mu}_{\ell}-\boldsymbol{\mu}_{\ell'})^{\sf{T}}\overline{\mathbf{C}}^{-1}(\boldsymbol{\mu}_{\ell}-\boldsymbol{\mu}_{\ell'}) \\
\end{split}
\end{equation}
where $\overline{\mathbf{C}}$ is the covariance matrix of the fused feature vector.

\section{Overview of AirBreath Sensing}
\label{sec:transceiver}
\begin{figure}[t!]
\centering
\includegraphics[width=1\columnwidth]{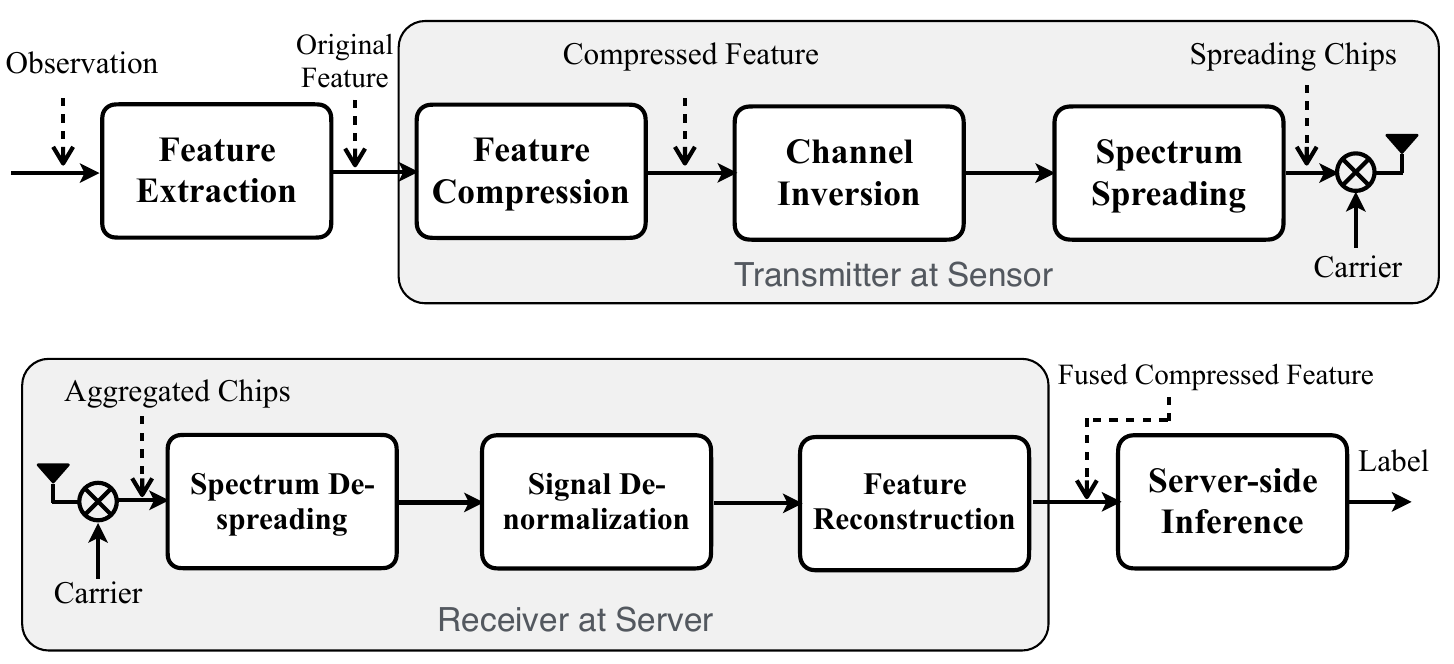}
\caption{The transceiver design of AirBreath sensing.}
\label{fig:framework_dig}
\vspace{-5mm} 
\end{figure}

As illustrated in Fig. \ref{fig:framework_dig}, the transceiver design of the AirBreath sensing that realizes the robust multi-view sensing is provided in the following subsections.

\vspace{-3mm}
\subsection{Transmitter Design}

As is shown in Fig. \ref{fig:framework_dig},
the transmitter is deployed at sensors and comprises three cascaded operations, i.e.,  feature compression, channel inversion, and spectrum spreading.

\subsubsection{Feature Compression}
The operation is to reduce the feature dimensions and create extra bandwidth for the latter operation of spread spectrum.
At sensor $k$, the feature vector $\mathbf{x}_k\in \mathbb{R}^{D}$ extracted from the observation using local pre-trained model is projected onto a low-dimensional subspace spanned by the compression matrix $\mathbf{P}\in \mathbb{R}^{D\times S}$\footnote{
Feature compression can be achieved using a \emph{deep neural network} (DNN), such as \emph{joint source and channel coding} (JSCC)~\cite{jianhao-JSCC,bourtsoulatze2019deep}, but this requires pre-training to handle dynamic channels, introducing additional latency and hindering real-time sensing.
In contrast, a compression matrix offers tractable analysis and avoids these limitations. }.
Here,  $S<D$ is the dimensionality of the subspace.
The compressed feature vector of sensor $k$, denoted as $\hat{\mathbf{x}}_k\in \mathbb{R}^{S}$, is computed by
\begin{equation}
    \hat{\mathbf{x}}_k = \mathbf{P}^{\sf{T}}\mathbf{x}_k.
\end{equation}
Note that the compression matrix $\mathbf{P}$ is determined at the edge server using the DG-based methods, which will be detailed in Sections \ref{sec:optimization} and \ref{sec:ext_CNN}.
To enable feature fusion in a common subspace via AirComp, the same compression matrix is assigned to all sensors.
The communication overhead of broadcasting $\mathbf{P}$ from the edge server to sensors is assumed to be negligible due to the sufficient transmit power at the edge server.
Since the feature statistics may change over inference rounds (e.g., features of different objects), feature normalization is required in each round to satisfy the power constraint. The normalized compressed feature vector is computed as
\begin{equation}
    \hat{\mathbf{x}}^{\sf nor}_k= \frac{\hat{\mathbf{x}}_k - \mu_{\sf nor} \mathbf{1}}{\sigma_{\sf nor}},
\end{equation}
where $\mathbf{1}$ is an all-one vector.
Here, 
$\mu_{\sf nor} 
$ and $ \sigma_{\sf nor}
$  denote the statistical mean and standard variance of $\hat{\mathbf{x}}_k$, respectively \cite{Zhiyan-AirPooling,liu2025semantic}. These two parameters are estimated 
and shared by all sensors and the edge server, enabling the normalized power, i.e., $\frac{1}{S}\mathbb{E}[ \Vert\hat{\mathbf{x}}_{k}^{\sf nor}\Vert^2] =1$.

\subsubsection{Channel Inversion} 
\label{sec:channel_inversion}

Fading channels across distributed sensors present challenges for AirComp-based feature fusion. To mitigate this, truncated channel inversion is employed to achieve signal amplitude alignment at the edge server. We consider block-fading channels, where the channel state $h_k$ remains constant during the feature transmission of sensor $k$. The corresponding transmit power $p_k$ is set as
\begin{equation}
  \label{CI_PC}
  p_k=
  \begin{cases}
\frac{\sqrt{P_{0}}h_k^*}{|h_k|^2},&|h_k|^2\geq h_{\sf th}, \\
    0,&|h_k|^2< h_{\sf th},
  \end{cases}
\end{equation}
where $P_{0}$ is the  signal-magnitude-alignment factor.
Given the interference power in \eqref{eq:rece_signal}, 
$\gamma_{\sf{sen}}=\frac{P_0}{P_I}$ denotes the receive SIR of each sensor after channel inversion.
The activation probability that measures the probability of successful feature uploading, denoted by $P_{\sf act}$, is then obtained as 
\begin{equation}
  \label{active-prob}
P_{\sf act}=\Pr(|h_k|^2\geq h_{\sf th})=\exp{(-h_{\sf th})}.
\end{equation}
We denote the subset of transmitting sensors by $\mathcal{K}=\{k\in\{1,2,\ldots,K\} \mid |h_k|^2\ge h_{\sf th}\}$.
The threshold $h_{\sf th}$ is chosen such that $\mathbb{E}[|p_k|^2]=  P_0 E_1(h_{\sf th})\le P_{\max}$, where $E_1(x)=\int_{x}^\infty \frac{\exp(-t)}{t} dt$ and $P_{\max}$ is the long-term  transmit power constraint of each sensor.

\subsubsection{Spectrum Spreading}
To suppress the interference, spectrum spreading expands the data bandwidth, denoted as $B_D(S)$, into the whole available bandwidth, denoted as $B_W$, using PN sequence. 
The durations of one feature symbol and one chip of the PN sequences are denoted as $T_D(S)=\frac{1}{B_D(S)}$ and $T_C=\frac{1}{B_W}$, respectively.
This gives the PN sequences comprising $G=\frac{T_D(S)}{T_C}$ chips, where $G$ is called the \emph{processing gain}.
The PN sequence here is represented by a vector $\mathbf{f}\in \mathbb{R}^{G}$, wherein each entry, say the $m$-th element $f_m\in \{+1,-1\}$,  is generated at the server through Bernoulli trials with the probability 0.5.
This ensures $\frac{1}{G} \mathbf{f}^{\sf{T}}\mathbf{f}=1$.
The PN sequences $\mathbf{f}\in \mathbb{R}^{G}$ are broadcast to all sensors for initialization of the spreader.
Using the scheduled PN sequences, the spreader transforms the normalized compressed feature vector into 
spreading chips, denoted as $\tilde{\mathbf{X}}_k\in \mathbb{R}^{S\times G}$, by computing
\begin{equation}
    \tilde{\mathbf{X}}_k =  \hat{\mathbf{x}}^{\sf nor}_k\mathbf{f}^{\sf{T}}.
\end{equation}
Here, the $(s,m)$-th element of $ \tilde{\mathbf{X}}_k$ is denoted as $\tilde{{X}}_{k,s,m}$ in \eqref{eq:rece_signal} representing the $m$-th chip of the $s$-th feature element of $\hat{\mathbf{x}}^{\sf nor}_k$.
At last, each sensor modulates the spreading chips onto the dedicated carrier and transmits them to the edge server.

\vspace{-3mm}

\subsection{Receiver Design}

As illustrated in Fig.~\ref{fig:framework_dig}, the receiver comprises three cascaded operations: spectrum despreading, signal denormalization, and feature reconstruction. 
The details of these operations, along with the subsequent server-side inference, are described below.

\subsubsection{Spectrum De-spreading}

This operation aims to mine the desired feature vector from the interference using the despreader.
We assume perfect synchronization between transmitter and receiver to ensure chip-level de-spreading.
At the edge server, the received signal is the superimposed waveforms due to the simultaneous transmission among sensors.
Let $\mathbf{Y}\in \mathbb{R}^{S\times G}$  denote received chips wherein each element  ${Y}_{s,m}$ represents the $(s,m)$-th chip obtained by \eqref{eq:rece_signal}.
By introducing the truncated channel inversion in Sec. \ref{sec:channel_inversion},
the output of the de-spreader, denoted as $\bar{\mathbf{y}}\in \mathbb{R}^{S}$, is expressed as
\begin{equation}
\begin{split}
    \bar{\mathbf{y}} & = \frac{1}{G }\mathbf{Y}\mathbf{f}
    = \frac{1}{G }\sum_{k\in \mathcal{K}} \sqrt{P_0} \hat{\mathbf{x}}^{\sf nor}_k\mathbf{f}^{\sf{T}}\mathbf{f}+\frac{1}{G}\mathbf{Z}\mathbf{f}\\
    &=\sqrt{P_0}\sum_{k \in \mathcal{K}}  \hat{\mathbf{x}}_k^{\sf nor} +\bar{\mathbf{z}},
\end{split}
\end{equation}
where $\mathcal{K} \subseteq\{1,2,\dots,K\}$ represents the set of active sensors;
$\bar{\mathbf{z}} \sim \mathcal{CN}\left(0,  \frac{P_I}{G}\mathbf{I}_S\right)$ represents the interference induced on the normalized compressed fused feature vector, attenuated by the processing gain $G$.

\subsubsection{Signal De-normalization}
In this operation, the effect of signal normalization and truncated channel inversion is eliminated. 
The output of signal de-normalization, denoted as $\hat{\mathbf{y}}$, is provided as
\begin{equation}
\label{eq:input_LC}
\begin{split}
    \hat{\mathbf{y}} &= 
    \frac{\sigma_{\sf nor}}{|\mathcal{K}|\sqrt{P_0}}
    \bar{\mathbf{y}}+\mu_{\sf nor}\mathbf{1} =\frac{1}{|\mathcal{K}|}\sum_{k\in\mathcal{K}} \mathbf{P}^{\sf T} \mathbf{x}_k +\hat{\mathbf{z}}.
\end{split}
\end{equation}
Here, $\hat{\mathbf{z}}\sim \mathcal{CN}\left(0,\frac{\sigma^2_{\sf nor}}{G |\mathcal{K}|^2 \gamma_{\sf sen}}\mathbf{I}_S\right) $ is the induced interference on the compressed feature subspace, where
$\gamma_{\sf sen}= \frac{P_0}{P_I}$ is the receive SIR of each sensor.

\subsubsection{Feature Reconstruction}
This operation is designed for a pre-trained CNN server-side classifier that requires an input feature dimensionality of $D$.
In this context, we  reconstruct $\hat{\mathbf{y}}$ into a $D$-dimensional feature vector, denoted as $\hat{\mathbf{y}}_{\sf cnn}\in \mathbb{R}^{D}$, given as $\hat{\mathbf{y}}_{\sf cnn}=\mathbf{P}\hat{\mathbf{y}}$,
where $\mathbf{P}\in \mathbb{R}^{D\times S}$ is the compression matrix that maps the compressed features onto the original feature space.

\subsubsection{Server-side Inference}
This operation aims to estimate the object label by feeding the output of the AirBreath sensing receiver into the server-side classifier.  The operations for the linear and CNN classifiers are discussed below.
For linear classification, the inferred label is obtained by solving the Mahalanobis distance minimization problem in \eqref{Maha_min_classifier}, given by
\begin{equation}
\label{eq:LC_Maha_min_Y}
    \hat{\ell} = \argmin_{\ell} z_\ell(\hat{\mathbf{y}}).
\end{equation}
Here, $z_\ell(\hat{\mathbf{y}})$ is the squared Mahalanobis distance of input feature vector $\hat{\mathbf{y}}$ from its centroid in the subspace spanned by $\mathbf{P}$, given as
$ z_\ell(\hat{\mathbf{y}})= ( \hat{\mathbf{y}} -\mathbf{P}^{\sf T} \boldsymbol{\mu}_{\ell}) ^{\sf{T}}{\hat{\mathbf{C}}}^{-1}( \hat{\mathbf{y}} -\mathbf{P}\boldsymbol{\mu}_{\ell})$,
where $\hat{\mathbf{C}}$ is the covariance matrix of $\hat{\mathbf{y}}$, given by
\begin{equation}
\label{eq:covariance_matrix}
\hat{\mathbf{C}}=
\frac{1}{|\mathcal{K}|}
\mathbf{P}^{\sf{T}}\mathbf{C}\mathbf{P} + \frac{\sigma^2_{\sf nor}}{G|\mathcal{K}|^2\gamma_{\sf sen}}\mathbf{I}_S. 
\end{equation}

For CNN classification, the reconstructed feature vector $\hat{\mathbf{y}}_{\sf cnn}$ is fed into the server-side classifier and provides the inference label $\hat{\ell}=\argmax_{c_\ell} \{c_1,\dots,c_\ell,\dots,c_L\},
$ where $c_\ell$ is the output confidence score of the $\ell$-th class using the pre-trained server-side CNN classifier $f_{\sf ser}(\hat{\mathbf{y}}_{\sf cnn})$.



\section{Analysis of the Compression-spreading Tradeoff}
\label{sec:funde-tradeoff}

In this section, we focus on linear classification and characterize the effects of the two key operations, i.e., feature compression and spread spectrum, on the receive DG, a tractable surrogate of sensing performance.
Under the constrained system bandwidth, we uncover a tradeoff between these two operations.
These findings motivate the optimization of AirBreath sensing strategies for CNN classification, detailed in Section \ref{sec:ext_CNN}.

\vspace{-3mm}

\subsection{Tractable Surrogate of Sensing Accuracy}

Deriving a closed-form expression for the sensing accuracy in \eqref{eq:sensing_acc} is intractable owing to the coupling among pairwise decision boundaries in multi-class classification. 
To enable analytical tractability, we adopt the received feature vector characterization in Lemma~\ref{lamma:compress_feature_dis} and employ the lower bound on sensing accuracy given in Lemma~\ref{Lemma:Sensing_ACC_LB}.
Specifically, given a feature vector following the GMM in \eqref{eq:data_dis}. 
The output  of 
AirBreath sensing receiver,
described by  \eqref{eq:input_LC}  also follows a GMM, as formalized in the Lemma \ref{lamma:compress_feature_dis}.

\begin{Lemma}[Distribution of Received Feature]
\label{lamma:compress_feature_dis}
Let the compression matrix $\mathbf{P}\in \mathbb{R}^{D\times S}$ satisfy $\text{rank}(P)=S$ and PN sequence comprise $G$ chips. The received feature vector $\hat{\mathbf{y}}$ in \eqref{eq:input_LC} follows GMM, given by
\begin{equation}
\label{eq:dis_pruned_feature}
     \hat{\bf{y}}\sim  \frac{1}{L} \sum_{\ell=1}^L \mathcal{N}\left(\mathbf{P}^{\sf{T}}\boldsymbol{\mu}_\ell, \hat{\mathbf{C}}\right),
\end{equation} where $\hat{\mathbf{C}}$  is the covariance matrix of $\hat{\mathbf{y}}$, given in \eqref{eq:covariance_matrix}.
    
\end{Lemma}

Given the linear classifier's input, $\hat{\mathbf{y}}$, sensing accuracy is lower-bounded by the minimum DG as shown in Lemma \ref{Lemma:Sensing_ACC_LB}

\begin{Lemma}[Lower Bound of Sensing Accuracy~\cite{ZW2024ultra-LoLa}]
\label{Lemma:Sensing_ACC_LB}
The accuracy $A$ of a linear classifier solving \eqref{eq:LC_Maha_min_Y} satisfies
\begin{equation}
\label{eq:accuracy_LB}
A \geq 1-(L-1)Q\left(\frac{\sqrt{\mathcal{G}_{\min}}}{2}\right),
\end{equation}
where $\mathcal{G}_{\min}=\min_{\ell\neq\ell'}\mathcal{G}_{\ell,\ell'}(\hat{\mathbf{y}})$ denotes the minimum DG across all class pairs.
\end{Lemma}

In \eqref{eq:accuracy_LB}, the lower bound of sensing accuracy is observed to be a monotone-increasing function of  $\mathcal{G}_{\min}$.
Since $\hat{\mathbf{y}}$ undergoes the AirBreath sensing transceiver, the minimum DG depends on the coupling effects of class pairs $(\ell_1,\ell_2)$, feature compression using $\mathbf{P}$ and spread spectrum with $G$, leading to the intractable analysis.
To simplify the analysis and keep this paper focused,  we investigate the effects of feature compression and spread spectrum on the DG between the closest class pair in the original feature space, given by
\begin{equation}
\label{eq:closet-class}
    (\ell,\ell')=\argmin_{\ell_1\neq\ell_2} (\boldsymbol{\mu}_{\ell_1}-\boldsymbol{\mu}_{\ell_2})^{\sf T} 
 \mathbf{C}^{-1}
 (\boldsymbol{\mu}_{\ell_1}-\boldsymbol{\mu}_{\ell_2}).
\end{equation}
In this context, we define the  DG  of $\hat{\mathbf{y}}$ between class $\ell$ and $\ell'$ as the surrogate of sensing accuracy, termed \emph{receive DG}, given in Definition \ref{def:receive DG}.

\begin{Definition}[Receive Discriminant Gain]
\label{def:receive DG}
The receive DG is defined as the DG between the closest class pair $(\ell,\ell')$ in \eqref{eq:closet-class}, given by
\begin{equation}
\label{eq:subspace-DG}
    \mathcal{G}(S,G)=\boldsymbol{\mu}_{\ell \ell'}^{\sf{T}}\mathbf{P}\hat{\mathbf{C}}^{-1}\mathbf{P}^{\sf T}\boldsymbol{\mu}_{\ell \ell'}.
\end{equation}
where  $\boldsymbol{\mu}_{\ell \ell'}=\boldsymbol{\mu}_{\ell}-\boldsymbol{\mu}_{\ell'}$ denotes the centroid difference between the closest classes $\ell$ and $\ell'$.

    
\end{Definition}

\subsection{Compression-spreading Tradeoff}

In this subsection, we consider a general compression matrix $\mathbf{P}\in \mathbb{R}^{D\times S}$, assuming only that it has full column rank $S$ without imposing any specific structural constraints.
Under this general setting, we characterize the effects of feature compression and spread spectrum on the receive DG, as detailed in Lemma \ref{Dim_Gain} and \ref{Proce_Gain}, respectively.

\begin{Lemma}[Effects of Feature Compression]
\label{Dim_Gain}
For fixed $G$, the receive DG monotonically increases with the subspace dimension $S$:
\begin{equation}
\mathcal{G}(S+1,G)\geq \mathcal{G}(S,G),\quad \forall S\in \{1,2,\cdots,D-1\}.
\end{equation}
\end{Lemma}
\begin{proof}
    (See Appendix~\ref{Proof_Dim_Gain}.)
\end{proof}

\begin{Lemma}[Effects of Spread Spectrum]
\label{Proce_Gain}
For a fixed compression matrix $\mathbf{P}$, receive DG monotonically increases with processing gain $G$:
\begin{equation}
\mathcal{G}(S,G+1)\geq \mathcal{G}(S,G),\quad \forall G\in \{1,2,\cdots,D-1\}.
\end{equation}
\end{Lemma}
\begin{proof}
    (See Appendix~\ref{Proof_Proce_Gain}.)
\end{proof}

\begin{Remark}[Compression-spreading Tradeoff]
\label{Theorem:Tradeoff}
Under the system bandwidth constraint $S \times G \leq D$, a tradeoff between feature compression and spread spectrum arises, as revealed by their monotonic effects in  Lemmas \ref{Dim_Gain} and \ref{Proce_Gain}. 
In particular, increasing the feature dimension $S$ enhances the receive DG and sensing accuracy, while the reduced processing gain $G$ degrades the receive DG due to stronger interference and lower feature quality. 
\end{Remark}

\section{Optimal Control of AirBreath Sensing}
\label{sec:optimization}
In this section, we optimize the breathing depth, a key control variable of AirBreath sensing, by balancing the compression-spreading tradeoff mentioned in Remark \ref{Theorem:Tradeoff}.
To this end, the explicit expression of the receive DG is first derived, followed by quantifying the optimal breathing depth to maximize the DG.

\vspace{-5mm}
\subsection{Receive Discriminant Gain}

To tractably optimize the found tradeoff, this subsection derives the explicit expression of receive DG.
We consider an importance-aware compression scheme guided by \emph{linear discriminant analysis} (LDA), which projects feature vectors onto a subspace that maximizes inter-class variance relative to intra-class variance~\cite{Friedman-Book-2009}.
Let the symmetric and \emph{positive definite} (PD)  intra-class covariance matrix $\mathbf{C}$ be diagonalized via eigenvalue decomposition~\cite{abdi2010principal} as
\begin{equation}
    \mathbf{C}=\mathbf{U}\Sigma_C \mathbf{U}^{\sf T}.
\end{equation}
Here, $\mathbf{U}=[\mathbf{u}_1,\dots,\mathbf{u}_d,\dots,\mathbf{u}_D]$ is an orthogonal matrix whose columns are eigenvectors of $\mathbf{C}$.
$\Sigma_C$ is the diagonal matrix comprising the eigenvalues 
\begin{equation}
   \Sigma_C=\diag \{ \lambda_1, \dots,\lambda_d, \dots,\lambda_D\}.
\end{equation}
Projecting the centroid difference $\boldsymbol{\mu}_{\ell \ell'}$ onto the eigenspace yields 
$\hat{\boldsymbol{\mu}}_{\ell \ell'}= \mathbf{U}^{\sf T}\boldsymbol{\mu}_{\ell \ell'}$, where the $d$-th element  $\gamma_d=\hat{\boldsymbol{\mu}}_{\ell \ell'}(d)$ represents the  inter-class separation along the  $d$-th eigenvector. 
In such a manner, the DG of $d$-th eigenvectors, i.e., $\mathbf{u}_d$, is defined as 
\begin{equation}
W_d = \frac{\gamma_d^2}{\lambda_d}.
\end{equation}
We then consider the DG-based compression matrix, provided in Definition \ref{def:DG-matrix}. 

\begin{Definition}[DG-based Compression Matrix]
    \label{def:DG-matrix}
   Rank the eigenvectors $\{\mathbf{u}_d\}$ by decreasing $W_d$, i.e., $W_1\geq W_2\geq \dots \geq W_D$.
   The DG-based compression matrix $\mathbf{P}$ retains the top-$S$  eigenvectors with the highest DG, given as
\begin{equation}
\label{eq:eign_projection}
    \mathbf{P} = [ \mathbf{u}_1, \mathbf{u}_2,\dots, \mathbf{u}_S].
\end{equation}
\end{Definition}

With such a compression scheme, Proposition \ref{Lemma:DG_DG} provides the explicit expression of receive DG.

\begin{Proposition}[Receive DG]
\label{Lemma:DG_DG}
Consider feature compression using the matrix in  Definition \ref{def:DG-matrix},
the resultant receive DG, denoted as $\hat{\mathcal{G}}(S,G)$, is given by
\begin{equation}
\label{eq:expicit_DG}
\begin{split}
\hat{\mathcal{G}}(S,G)= |\mathcal{K}| \sum_{d=1}^S \frac{  W_d}{\hat{\sigma}^2  /(G\lambda_d)+1}
\end{split} 
\end{equation}
where  $\hat{\sigma}^2=\frac{\sigma^2_{\sf nor} }{|\mathcal{K}|\gamma_{\sf sen}}$ and $W_d$ is the DG of $\mathbf{u}_d$.

\end{Proposition}
\begin{proof}
    (See Appendix~\ref{Proof_DG_DG}.)
\end{proof}

By projecting the feature vector onto 
the eigenspace of covariance matrix $\mathbf{C}$, Proposition \ref{Lemma:DG_DG} demonstrates that $\hat{\mathcal{G}}(S,G)$  is a monotone increasing function of feature dimension $S$ (i.e., select more eigenvectors) and processing gain $G$, respectively. Serving as the objective, the receive DG in \eqref{eq:expicit_DG} is maximized to derive the optimal solution in the next subsection.

\vspace{-3mm}

\subsection{Optimal Breathing Depth}

To maximize the sensing accuracy, this subsection optimizes the breathing depth, a key control variable that synchronizes the effects of feature compression and spread spectrum on sensing accuracy under bandwidth constraints.
First, the maximization of receive DG in \eqref{eq:expicit_DG} is  formulated as
\begin{subequations}
\label{prob:max_DG}
\begin{align}
        \max_{S,G} & \quad  \hat{\mathcal{G}}(S,G)\label{eq:opt_obj} \\  \mathrm{s.t.} & \quad   G\times S\leq D, \label{eq:e2e_latency}\\
            & \quad S, G\in \{1,,\dots, D\},
\end{align}
\end{subequations}
where \eqref{eq:e2e_latency} represents the constrained communication overhead resulting from limited bandwidth resources.
Deriving the optimal solution is challenging due to the general distribution of parameters $\{\gamma_s\}$ and $\{\lambda_s\}$ over feature dimensions.
For  tractability, we consider maximizing the lower bound of $ \hat{\mathcal{G}}(S,G)$, given by
\begin{equation}
\label{eq:phi_S_G}
\begin{split}
       \hat{\mathcal{G}}(S,G)
      & \geq \phi(S,G)=  \frac{|\mathcal{K}| \sum_{d=1}^S  W_d }{\hat{\sigma}^2  /(G \lambda_{\min})+1},\\
\end{split}
\end{equation}
where $\hat{\sigma}^2=\frac{\sigma^2_{\sf nor} }{|\mathcal{K}|\gamma_{\sf sen}}$, and $\lambda_{\min}=\min_{d\in \{1,2,\dots,D\}}\lambda_d$ is the minimum eigenvalue.

Next, we define the breathing depth of AirBreath sensing in Definition \ref{def:BD} to synchronize the two opposite effects within the constrained bandwidth in \eqref{eq:e2e_latency}.

\begin{Definition}[Breathing Depth] 
\label{def:BD}
In the considered ISEA systems with bandwidth and latency constraints, it is essential to maintain the total number of transmitted chips unchanged after spectrum breathing, such that $S \times G = D$ (i.e., equivalent to scenarios without spectrum breathing).
Under this constraint, the tradeoff between feature compression and spread spectrum is regulated by the feature subspace dimension $S=\frac{D}{G}$.
For clarity, $S$ is referred to as the breathing depth, which serves as the primary control variable in AirBreath sensing.
\end{Definition}

\begin{figure}[t!]
\centering
\subfigure[SIR $=-3$ dB]{
\includegraphics[width=0.45\columnwidth]{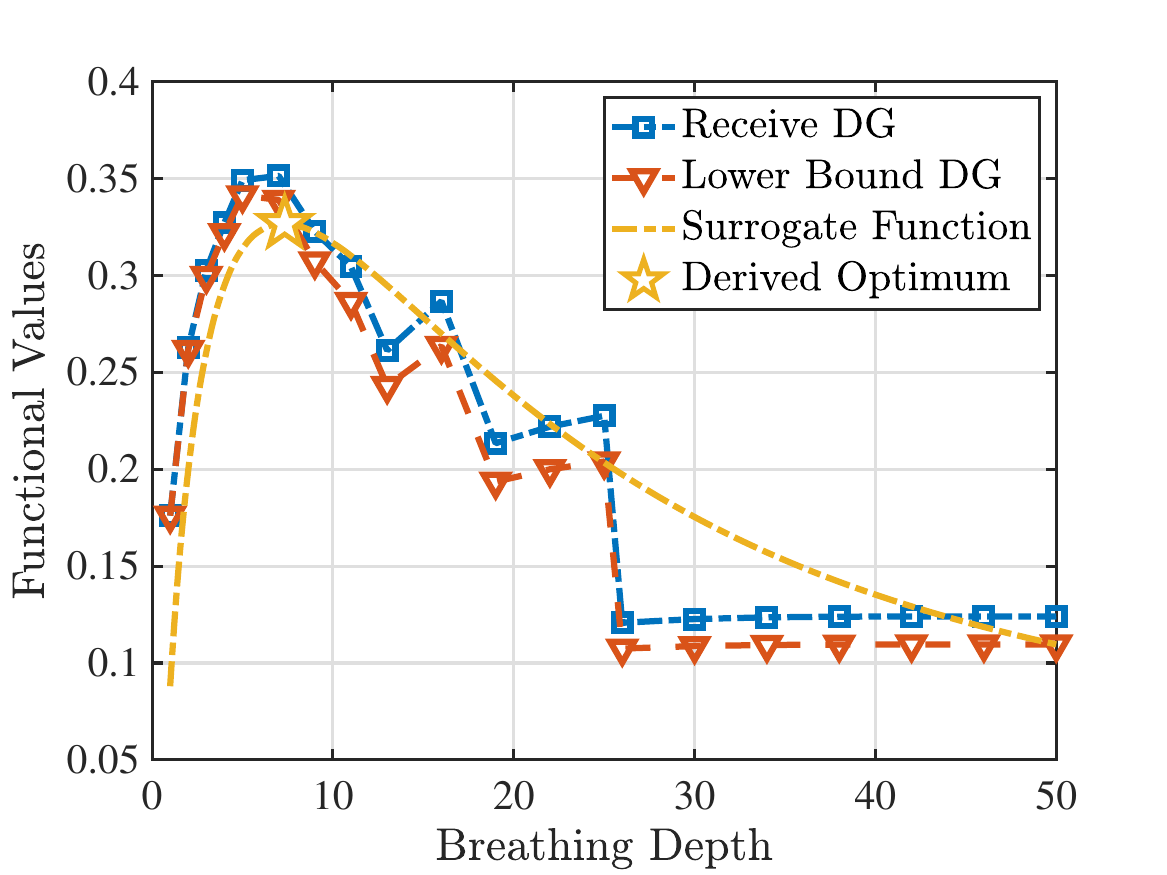}}
\subfigure[ SIR $=7$ dB]{
\includegraphics[width=0.45\columnwidth]{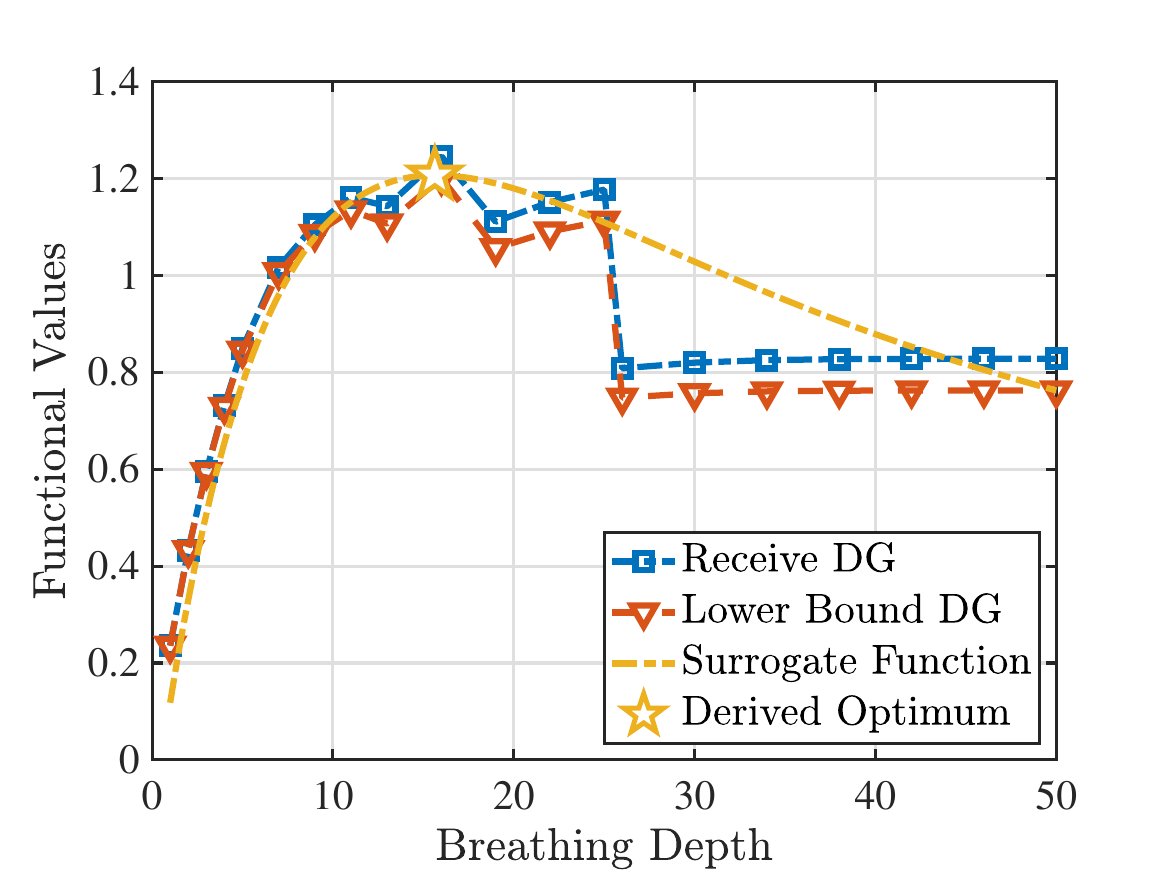}}
\caption{In the case of linear classification with constrained bandwidth, the comparison among receive DG in \eqref{eq:expicit_DG}, its lower bound  $\hat{\phi}(S)$, and the proposed surrogate $\tilde{\phi}(S)$.}
\label{fig:tradeoff}
\vspace{-5mm} 
\end{figure}

Replacing the $G$ in \eqref{eq:phi_S_G} with a function of breathing depth, i.e., $G_S=\left\lfloor\frac{D}{S}\right \rfloor$,  $\phi(S,G)$ becomes a discrete function of breathing depth $S\in\{1,2,\dots,D\}$, denoted as $\hat{\phi}(S)$, given by
\begin{equation}
\label{eq:hat{phi}(S)}
    \hat{\phi}(S)=\frac{|\mathcal{K}|\sum_{d=1}^{S}  W_d}{\hat{\sigma}^2 /(G_S \lambda_{\min})+1}.
\end{equation}
The optimization problem \eqref{prob:max_DG} is transformed into maximizing $ \hat{\phi}(S)$,  formulated by
\begin{subequations}
\label{prob:hat{phi}(S)}
\begin{align}
        \max_{\omega} & \quad      \hat{\phi}(S)  \\  \mathrm{s.t.}    
            & \quad S\in \{1,2,\dots,D\}.
\end{align}
\end{subequations}

To show the uniqueness of the optimal breathing depth in Problem \eqref{prob:hat{phi}(S)}, we approximate $\hat{\phi}(S)$ as a continuous function of $\omega$ 
 by considering two approximations.
 The one is to approximate maximum processing gain as  $G_S\approx \frac{D}{S}$ by allowing them to take the non-integer values.
The other is to approximate $\sum_{d=1}^{S}  W_d$ in \eqref{eq:hat{phi}(S)} as a continuous function of $S$, denoted as $\psi(S)$, given by~\cite{zw2025AIoutage}
\begin{equation}
    \sum_{d=1}^{S}  W_d\approx\int_{0}^{S} g(t) dt\triangleq\psi(S),
\end{equation}
where $g(t)$ is a continuous and differentiable function in the interval
 $t \in [0,S]$:
\begin{equation}
    g\left(t\right)=
        \frac{W_d-W_{d+1}}{2}\cos(\pi (t-d+1))+\frac{W_d+W_{d+1}}{2},
\end{equation}
otherwise, $g\left(t\right)=0, \forall t \notin [0,S]$.
Here, $W_d$ is the DG of the $d$-th feature dimension.
This leads to the approximation of \eqref{eq:hat{phi}(S)}, denoted as $\tilde{\phi}(S)\approx \hat{\phi}(S)$,  written as
\begin{equation}
\label{eq:LC_surrogate}
\tilde{\phi}(S) = \frac{|\mathcal{K}|\psi(S)}{\tilde{\sigma}^2  S+1}, \quad S\in[1,D]
\end{equation}
where $\tilde{\sigma}^2=\frac{\sigma^2_{\sf nor}}{D|\mathcal{K}|\lambda_{\min}\gamma_{\sf sen}}$.

 In Fig.~\ref{fig:tradeoff}, we compare the proposed surrogate function $\tilde{\phi}(S)$ with the original receive DG in \eqref{eq:expicit_DG} and its lower bound $\hat{\phi}(S)$ under a bandwidth constraint. The surrogate closely approximates its ground-truth counterparts and accurately captures the optimal points, exhibiting negligible approximation error. The compression–spreading tradeoff is also evident, as all three curves initially increase with breathing depth, reach a peak, and then decline. The observed fluctuations in the receive DG and its lower bound are attributed to the integer constraints on processing gain and breathing depth.

By replacing the objective with its approximation,  problem \eqref{prob:hat{phi}(S)} can be reformulated as 
\begin{subequations}
\label{prob:tilde{phi}(S)}
\begin{align}
        \max_{S} & \quad      \tilde{\phi}(S)  \\  \mathrm{s.t.}    
            & \quad S\in \{1,2,\dots, D\}.
\end{align}
\end{subequations}

Problem \eqref{prob:tilde{phi}(S)} has the unique optimal solution that maximizes the receive DG of the AirBreath sensing system, which is provided in Proposition \ref{prop:OPT_feature}.

\begin{Proposition}[Optimal Breathing Depth]
\label{prop:OPT_feature}
Let 
\begin{equation}
   \zeta(x) = \psi'(x)\left(\tilde{\sigma}^2 x + {1}\right) - \tilde{\sigma}^2 \psi(x).
\end{equation}
denote the function of $x\in[1,D]$. The optimal breathing depth that solves problem \eqref{prob:tilde{phi}(S)}, denoted as $S^*$, is then
\begin{equation}
\label{eq:S*}
    S^*= \left \lfloor x^* \right\rceil_{ \tilde{\phi}(\cdot)},
\end{equation}
where the rounding operator $\lfloor x \rceil_{{ \tilde{\phi}(\cdot)}}$ is equal to $\lfloor x \rfloor$ if $ \tilde{\phi}(\lfloor x \rfloor)\geq \tilde{\phi}(\lceil x \rceil) $, and is otherwise equal to $\lceil x \rceil$. 
The value $x^*$ is given by
\begin{equation}
x^*  = \left\{x| \zeta(x)=0, x\in [1,D]  \right\},
\end{equation}
if $\zeta(1)\cdot \zeta (D)<0$ holds; otherwise $S^*  = \argmax_{x\in \{1,D \}}\tilde{\phi}(x)$.

\end{Proposition}
\begin{proof}
    (See Appendix~\ref{proof_prop:OPT_feature}.)
\end{proof}

Fig. \ref{fig:channel_optimal} shows the dependence of the optimal breathing depth, computed using Proposition \ref{prop:OPT_feature}, on the channel state. The optimal breathing depth increases with higher received SIR or a larger number of active sensors, as both reduce interference power. Consequently, the optimized compression–spreading tradeoff favors a smaller processing gain (i.e., more feature dimensions) to enhance sensing accuracy. These results highlight the channel-adaptive capability of AirBreath sensing under dynamic wireless conditions.

\section{Extension to CNN Classification}
\label{sec:ext_CNN}

Leveraging insights from linear classification, this section derives an explainable surrogate function for the CNN classification and subsequently optimizes the breathing depth.

\begin{figure}[t!]
\centering
\subfigure[Receive SIR vs Optimal Breathing Depth]{
\includegraphics[width=0.45\columnwidth]{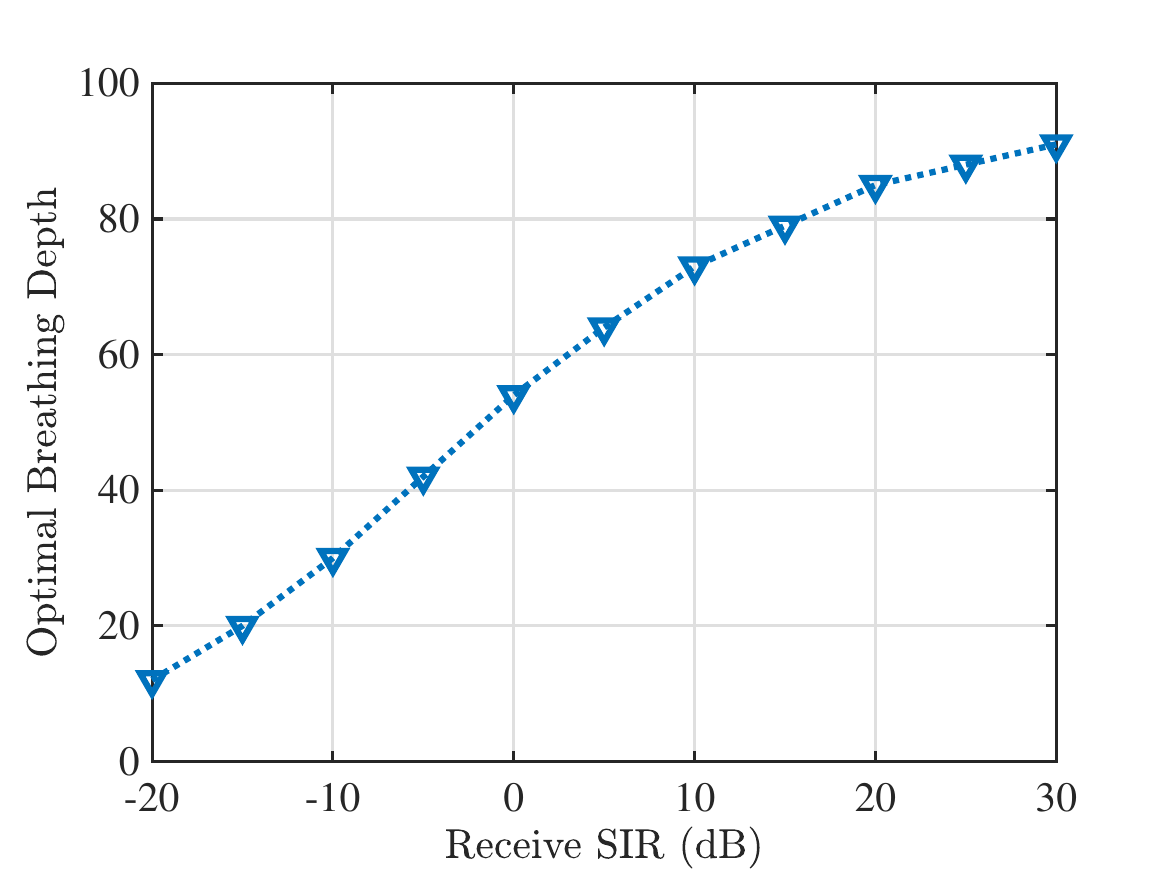}}
\subfigure[Active Sensors vs Optimal Breathing Depth]{
\includegraphics[width=0.45\columnwidth]{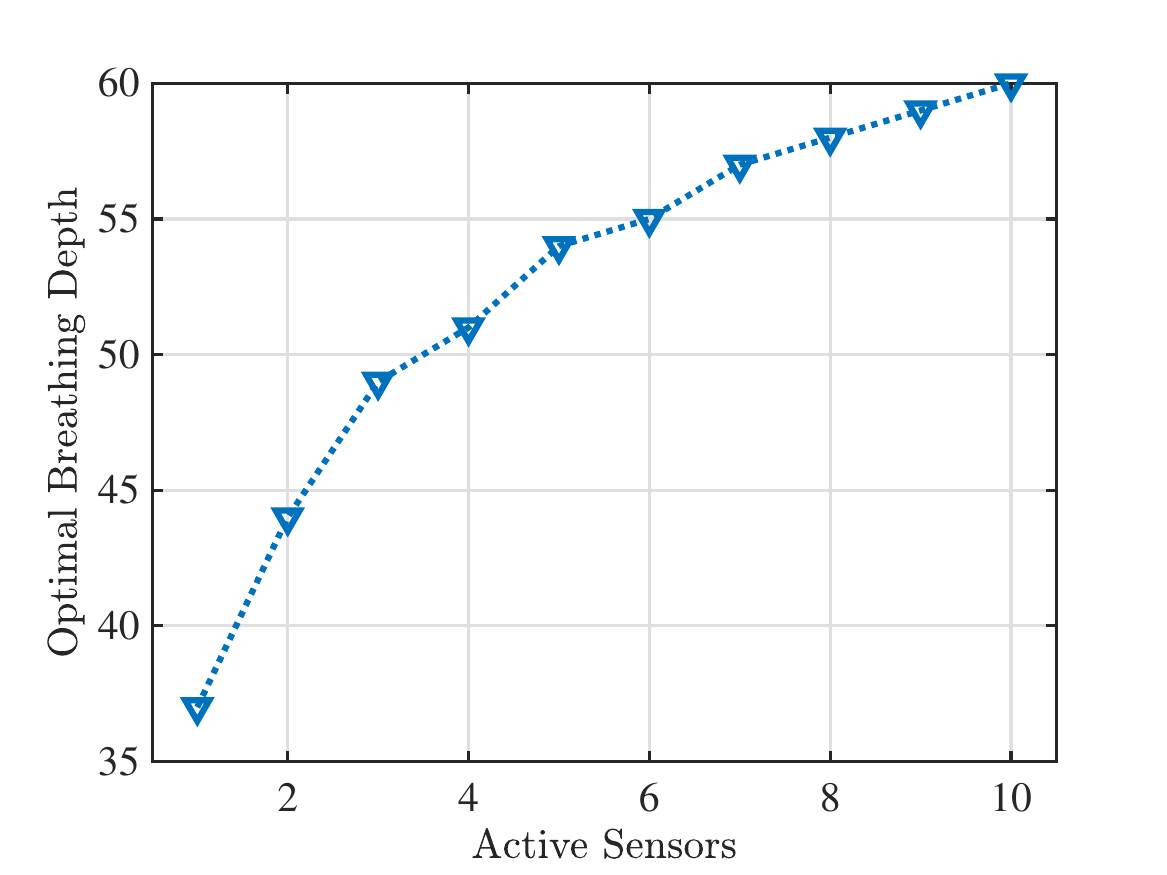}}
\caption{The effects of channel condition on the optimal breathing depth.}
\label{fig:channel_optimal}
\vspace{-5mm} 
\end{figure}

\vspace{-3mm}
\subsection{Surrogate of Sensing Accuracy}
Unlike linear classification, the nonlinearity of the CNN classifier complicates the modeling of sensing accuracy with respect to the received feature vector. 
To address this, we develop a CNN  surrogate function by leveraging the insights from the analysis on linear classification, given in Remark \ref{rem:Lesson}.

\begin{Remark}[Insight from Linear Classification]
\label{rem:Lesson}
The CNN performance surrogate is informed by two key insights from linear classification: 1) Sensing accuracy rises monotonically with both processing gain and feature dimension, as shown in Lemmas~\ref{Dim_Gain} and~\ref{Proce_Gain}; 
2) Mapping sensing accuracy to its DG domain enables a linear combination of compression and spreading effects on the incremental DG, denoted $\mathcal{G}_{\Delta}$,  computed by increasing breathing depth by one, mathematically,
\begin{equation}
\label{eq:incremental_DG}
\begin{split}
\mathcal{G}_{\Delta}=&~\mathcal{G}\left(S+1,G_{S+1}\right)-\mathcal{G}\left(S,G_S\right)\\
=&~\underbrace{\mathcal{G}\left(S+1,G_{S+1}\right)-\mathcal{G}\left(S,G_{S+1}\right)}_{\text{compression effect}}\\
&~+\underbrace{\mathcal{G}\left(S,G_{S+1}\right)-\mathcal{G}\left(S,G_S\right)}_{\text{spreading effect}},
\end{split}
\end{equation}
where $G_S = \left\lfloor\frac{D}{S}\right\rfloor$ and $G_{S+1} = \left\lfloor\frac{D}{S+1}\right\rfloor$ are the maximum processing gains under the bandwidth constraints.
\end{Remark}

Inspired by Remark \ref{rem:Lesson}, we consider the relationship between sensing accuracy and CNN DG given in Definition \ref{Def: CNN DG}.

\begin{Definition}[CNN Discriminant Gain~\cite{zw2025AIoutage}] \label{Def: CNN DG} Given the sensing accuracy of a CNN classifier, denoted by $A_{\sf cnn}$, the corresponding CNN DG, denoted by $\mathcal{G}_{\sf{cnn}}$, is defined as 
\begin{equation} \label{A2DG} \mathcal{G}_{\sf{cnn}} = \beta Q^{-1}\left(\alpha(1 - A_{\sf cnn})\right), \end{equation} where $\alpha$ and $\beta$ are hyperparameters. 
\end{Definition}

\begin{figure}[t!]
\centering
\subfigure[Effects of Spread Spectrum (SIR = $6.9$dB)]{\label{spread-effects}
\includegraphics[width=0.45\columnwidth]{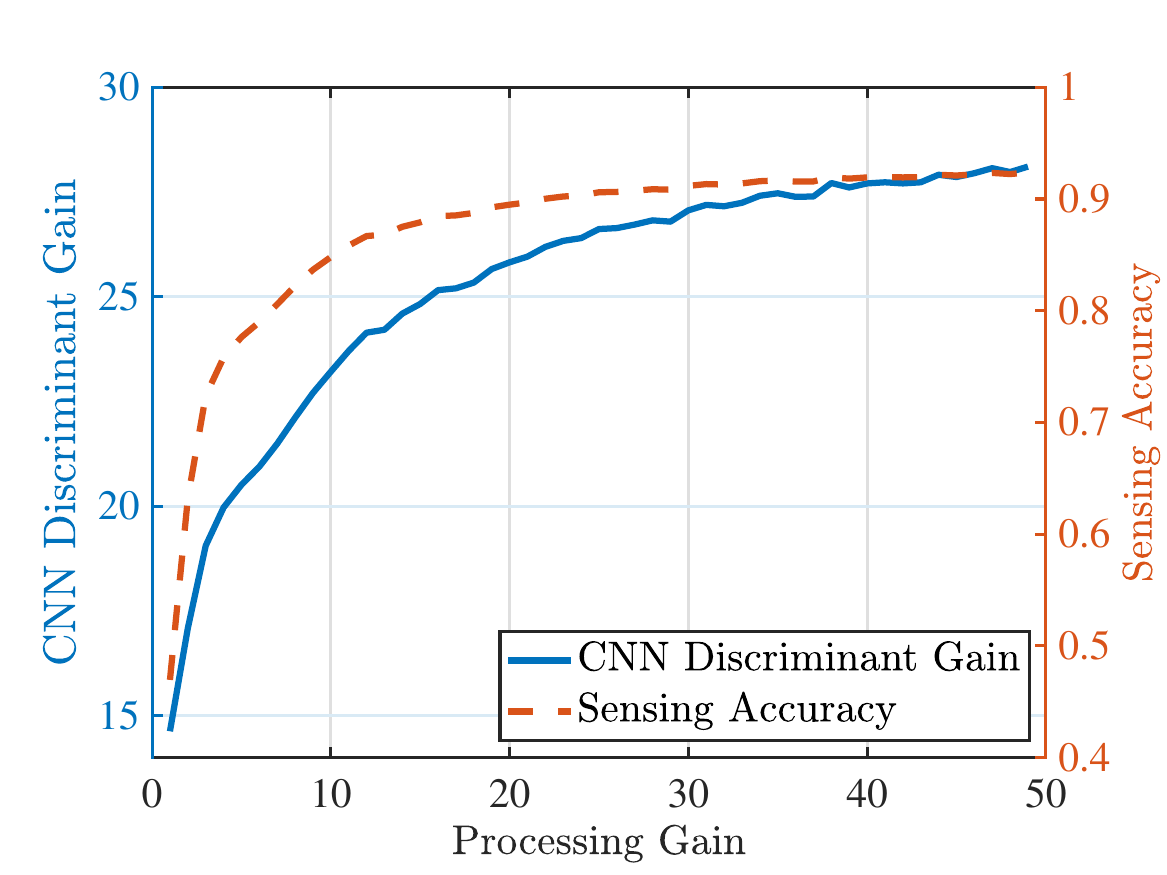}}
\subfigure[Effects of Feature Compression]{\label{compre-effects}
\includegraphics[width=0.45\columnwidth]{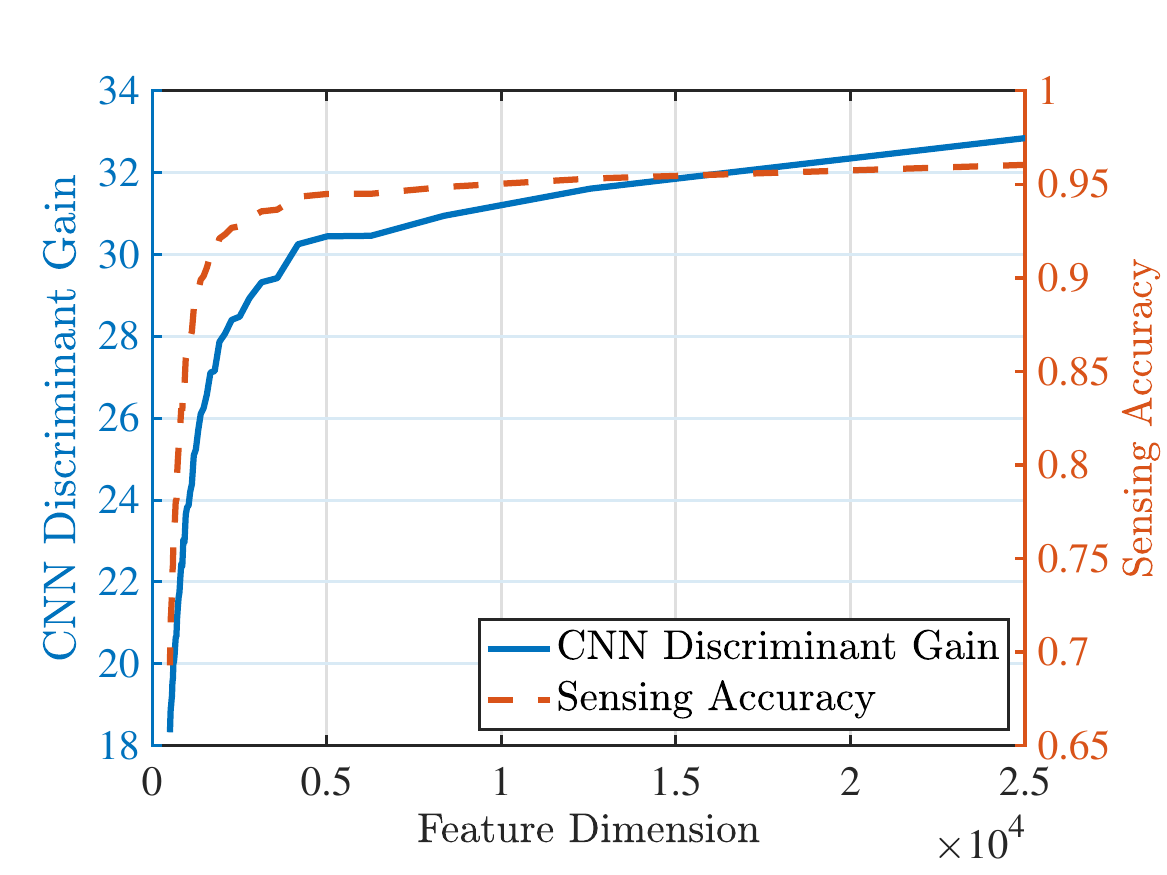}}
\caption{The estimated functions  $\mathcal{G}^{\sf comp}_{\sf cnn}(S)$ and $\mathcal{G}^{\sf spre}_{\sf cnn}(G)$ over feature dimension $S$ and processing gain $G$, respectively, on the training dataset. }
\label{fig:compression-spread-effects}
 \vspace{-5mm} 
\end{figure}

Leveraging Definition \ref{Def: CNN DG} and inspired by \eqref{eq:incremental_DG}, we propose a CNN surrogate function, denoted by $\hat{\mathcal{G}}_{\sf{cnn}}(S,G)$, to maximize sensing accuracy. This surrogate is expressed as the summation of two associated DGs, given by
\begin{equation}
\label{eq:cnn_surrog}
   \hat{\mathcal{G}}_{\sf{cnn}}(S,G) = \mathcal{G}^{\sf comp}_{\sf cnn}(S) + \mathcal{G}^{\sf spre}_{\sf cnn}(G).
\end{equation}
Here, $\mathcal{G}^{\sf spre}_{\sf cnn}(G)$ is a monotone-increasing function of processing gain $G$ as shown in Fig. \ref{spread-effects}, characterizing the effects of spread spectrum on CNN DG.
This function is obtained by Algorithm \ref{Alg}.
On the other hand, $\mathcal{G}^{\sf comp}_{\sf cnn}(S)$ quantifies the effect of feature compression on CNN DG, represented by a monotone-increasing function of $S$, as illustrated in Fig.~\ref{compre-effects}.
This function is obtained by evaluating the sensing accuracy and the associated CNN DG under an importance-aware feature compression scheme.
The importance of CNN features is measured statistically by their inter-class discriminability, as elaborated as follows.
Consider a training dataset comprising a set of object classes $\{\mathcal{C}_1, \dots, \mathcal{C}_\ell, \dots, \mathcal{C}_L\}$, where $\mathcal{C}_\ell$ denotes the set of objects belonging to class $\ell$.
Let $\mathbf{x}_{j,k}$ denote the feature tensor extracted from the $k$-th observation of the $j$-th object.
The statistical mean of class $\ell$ is computed as $\boldsymbol{\mu}_{{\ell}}^{\sf cnn}=\frac{1}{|\mathcal{C}_{\ell}|K}\sum_{j\in\mathcal{C}_{\ell}} \sum_{k=1}^K \mathbf{x}_{j,k}$.
With such statistics estimated on the training dataset, the  importance of the $d$-th feature dimension, denoted by $I_d$, is obtained offline, given by
\begin{equation}
\label{eq:feature_importance}
    I_d=\frac{2}{L(L-1)}\sum_{\ell_1=1}^L\sum_{\ell_2>\ell_1}\sqrt{(\mu^{\sf cnn}_{\ell_1}(d)-\mu^{\sf cnn}_{\ell_2}(d))^2}.
\end{equation}
Based on the element-wise feature importance, the sensing accuracy of CNN classification is computed by inputting the top-$S$ most important features into the classifier. The corresponding CNN DG, $\mathcal{G}^{\sf comp}_{\sf cnn}(S)$, is then obtained using \eqref{A2DG}.
 The compression-spreading tradeoff stems from the monotonic dependencies of the CNN DG on both feature compression and spread spectrum, as illustrated in Fig.~\ref{fig:compression-spread-effects}, meanwhile limited by the system bandwidth. This observation is consistent with the tradeoff described in Remark \ref{Theorem:Tradeoff} for  linear classification.

\begin{algorithm}[t]
 \caption{Effect of Spread Spectrum on CNN DG}
 \begin{algorithmic}[1]
 \renewcommand{\algorithmicrequire}{\textbf{Input:}}
  \REQUIRE  Number of observations $\mathcal{K}$, Receive SIR $\gamma_{\sf sen}$, and Processing gain $ G\in \{1,2, \dots, D\}$
\STATE{Initialisation}: Training datasets and well-trained model;
\FOR{Procesing gain: $G\in \{1,2, \dots, D\}$}
  \FOR{ Data samples in training dataset}
   \STATE Randomly select a batch of observations with size of $|\mathcal{K}|$ from the training dataset of a common object;
    \STATE   Extract feature vector $\{\mathbf{x}_1,\mathbf{x}_2,\dots,\mathbf{x}_{|\mathcal{K}|}\}$;
    \STATE Compute the fused feature vector: $\overline{\mathbf{x}}= \frac{1}{|\mathcal{K}|}\sum_{k=1}^{|\mathcal{K}|} \mathbf{x}_k$;
  \STATE Emulate the effects of channel-induced interference on fused feature vector: $ \tilde{\mathbf{x}} = \overline{\mathbf{x}}+ \mathbf{z}_{\sf cnn}$,
  where $\mathbf{z}_{\sf cnn}\sim \mathcal{N}(0,\frac{\sigma^2_{\sf nor}}{|\mathcal{K}|^2G\gamma_{\sf sen}} \mathbf{I}_D)$;
  \STATE Infer label using $\tilde{\mathbf{x}}$;
    \ENDFOR
    \STATE Compute the sensing accuracy $A_{\sf cnn}$;
    \STATE Compute the CNN DG $   \mathcal{G}^{\sf spre}_{\sf cnn}(G)$ using \eqref{A2DG};
\ENDFOR
  \RETURN $\mathcal{G}^{\sf spre}_{\sf cnn}(G)$;
 \end{algorithmic} 
 \label{Alg}
 
 \end{algorithm}

\vspace{-3mm}
\subsection{Optimal Breathing Depth for CNN Classification}

Building upon the CNN surrogate function presented in \eqref{eq:cnn_surrog}, the optimization problem for the CNN case  under a bandwidth constraint is formulated as follows:
\begin{subequations}
\label{prob:max_DG_cnn}
\begin{align}
        \max_{S,G} & \quad  \hat{\mathcal{G}}_{\sf cnn}(S,G)\label{eq:cnn_opt_obj} \\  \mathrm{s.t.} & \quad   G\times S\leq D, \label{eq:cnn_bandwith}\\
            & \quad S,G\in \{1,2,\dots, D\}.
\end{align}
\end{subequations}
Since the CNN surrogate increases with an increase in either feature dimensions $S$ or the processing gain $G$, the bandwidth constraint \eqref{eq:cnn_bandwith} should be held in equality. 
By substituting the processing gain with its maximum $G_S=\lfloor \frac{D}{S} \rfloor$, the optimization can be reformulated as the maximization of the univariate function $\hat{\mathcal{G}}_{\sf cnn}(S,G_S)$ of breathing depth $S$, where $G_S=\left\lfloor \frac{D}{S}\right\rfloor$.
Conditioned on the active sensor $\mathcal{K}$ and the receive SIR $\gamma_{\sf sen}$, the optimal breathing depth is given by
\begin{equation}
\label{eq:opt_dim_cnn}
    S^*=\argmax_{S\in\{1,2,\dots,D\}}\hat{\mathcal{G}}_{\sf cnn}(S,G_S).
\end{equation}
The optimal solution can be efficiently determined via a bisection search over $S\in \{1,2,\dots D\}$ that maximizes $\hat{\mathcal{G}}_{\sf cnn}(S,G_S) $, with the complexity of $\mathcal{O}(\log D)$.

\section{Experimental Results}
\label{sec:experiments}

\subsection{Experimental Settings}

Unless stated otherwise, the experiments are conducted under the following default settings.

\subsubsection{System Configuration}
We consider an AirComp-based multi-view sensing system comprising $K=10$ distributed sensors, collaboratively performing object recognition over $N$ inference rounds. 
Consider round $n\in\{1,2,\dots,N\}$,  the object is randomly selected from a predefined object set.
Each sensor extracts features and transmits them to the edge server over  Rayleigh fading channels, where the complex channel gain is modeled as $h_{k} \sim \mathcal{CN}(0,1)$ and varies independently across rounds. 
To cope with fading channels and realize amplitude alignment at edge server, truncated channel inversion is employed, with a threshold $h_{\sf th} = 0.1054$. This results in a sensor activation probability of $P_{\sf act} = 0.9$.
In round $n$, spread spectrum is performed using a PN sequence consisting of $G_n=\lfloor\frac{D}{S^*_n}\rfloor$ chips, where $S_n^*$ is the optimal breathing depth computed by Proposition \ref{prop:OPT_feature}  for linear classification and by \eqref{eq:opt_dim_cnn} for CNN classification.
Each chip is generated by the edge server using i.i.d. Bernoulli trials from the set $\{+1, -1\}$ and broadcast to distributed sensors. 
The chip duration is set to $T_C = 45.45$ ns, corresponding to a chip rate of 22 MHz, following the \emph{direct-sequence spread spectrum} (DSSS) standard in IEEE 802.11b~\cite{IEEE802.11b}.
The PN sequence is independently refreshed in each round for the concern of communication security.
The sensing accuracy is measured as the proportion of correctly classified instances relative to the total number of $N$ rounds.

\subsubsection{Classifier Settings}
Two classifiers and the corresponding datasets are detailed as follows.
\begin{itemize}
    \item \textbf{Linear Classification on Synthetic GMM Data:} 
The feature vector of linear classification is generated according to the GMM defined in \eqref{eq:data_dis}.
We consider a binary classification comprising two centroids with decreasing centroid difference over $D=50$ feature dimensions.
One of the  centroids is set as $\boldsymbol{\mu}_1=[1,(1-\frac{1}{D})^2,(1-\frac{2}{D})^2,\dots,\frac{1}{D^2}]$ and the other is $\boldsymbol{\mu}_2=-\boldsymbol{\mu}_1$. The covariance matrix is given by a diagonal matrix with increasing elements over dimension, given by  $\diag \{1+\frac{1}{D}, 1+\frac{2}{D},\dots,2  \}$.
 The covariance matrix is given by a diagonal matrix with increasing elements over dimensions, given by  $\diag \{1+\frac{1}{D},1+\frac{2}{D},\dots,2 \}$.
In this context, to select the Top-$S$ dimensions with the highest DGs, the compression matrix is set as $\mathbf{P}=[\mathbf{I}_S,\mathbf{0}_{(D-S)\times S}]^{\sf T}$ where $\mathbf{0}_{(D-S)\times S}$ denotes an all-zero matrix.
The inferred label is obtained by feeding the output of the AirBreath sensing into the classifier in \eqref{eq:LC_Maha_min_Y}.

 \item  \textbf{CNN Classification on ModelNet Dataset:}
In the case of the CNN classifier, we employ the widely used ModelNet dataset~\cite{ModelNet-Ref}, which provides multi-view object observations (e.g., a bed or a car), and implement the CNN architecture based on the VGG16 model~\cite{simonyan2015deep}. 
Following the approach in~\cite{Zhiyan-AirPooling}, the VGG16 model is partitioned into a feature extractor, executed on the device, and a classifier network, executed on the server. 
The resulting CNN architecture is trained with average pooling to recognize objects belonging to $L=20$ popular classes of ModelNet.
To perform feature extraction, the device randomly selects $K$ observations of the same class from the dataset and processes them through the feature extractor. 
The resulting feature tensor of each sensor has a dimension of $D=512\times 7 \times 7$ given the input image size of ModelNet of $244\times 3\times 3$.
Based on \eqref{eq:feature_importance}, the top $S^*$ features with the highest importance values are preserved.
Finally, the received feature tensor is reconstructed and passed to the server-side classifier to generate the inferred label.

\end{itemize}

\begin{figure}[t!]
\centering
\subfigure[Linear Classifier (SIR=$-23$ dB)]{\label{fig:lc_tradeoff_feature}
\includegraphics[width=0.45\columnwidth]{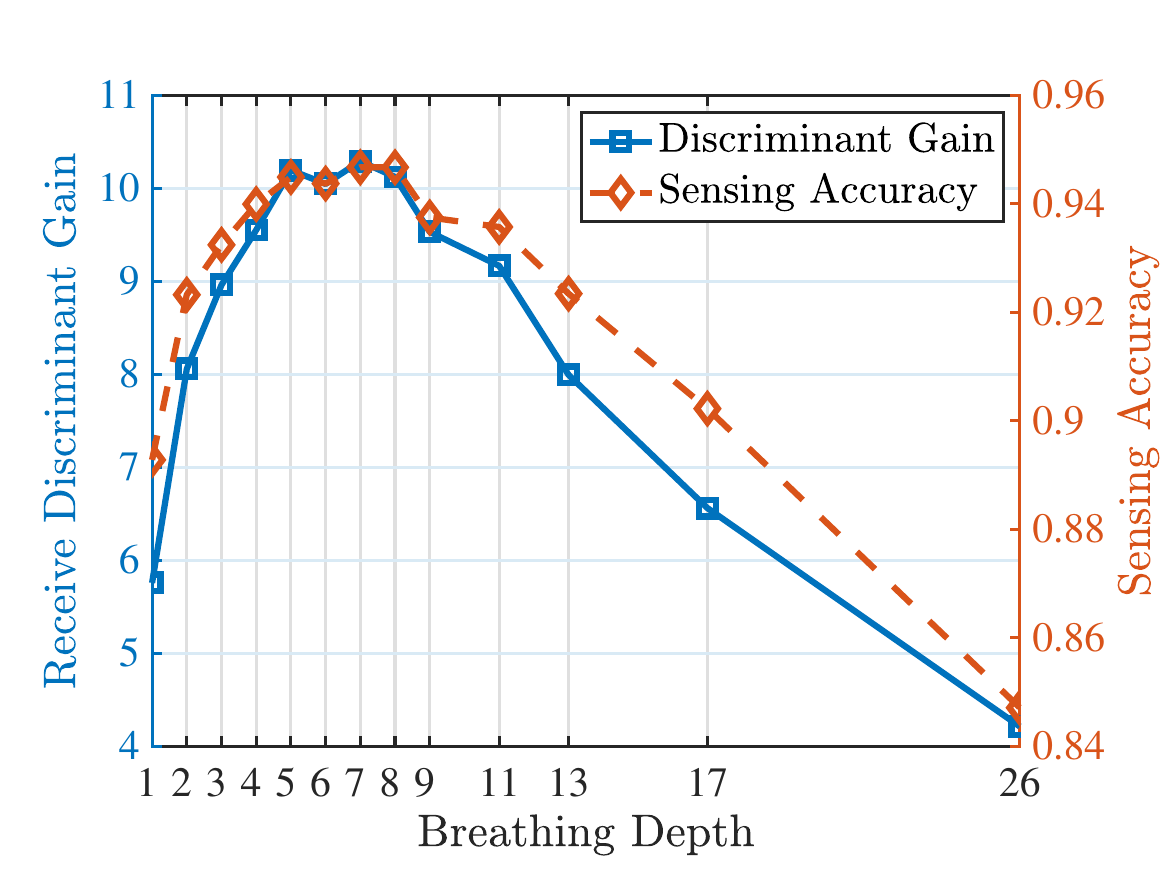}}
\subfigure[CNN Classifier (SIR=$6.9$ dB)]{\label{fig:cnn_tradeoff_fea}
\includegraphics[width=0.45\columnwidth]{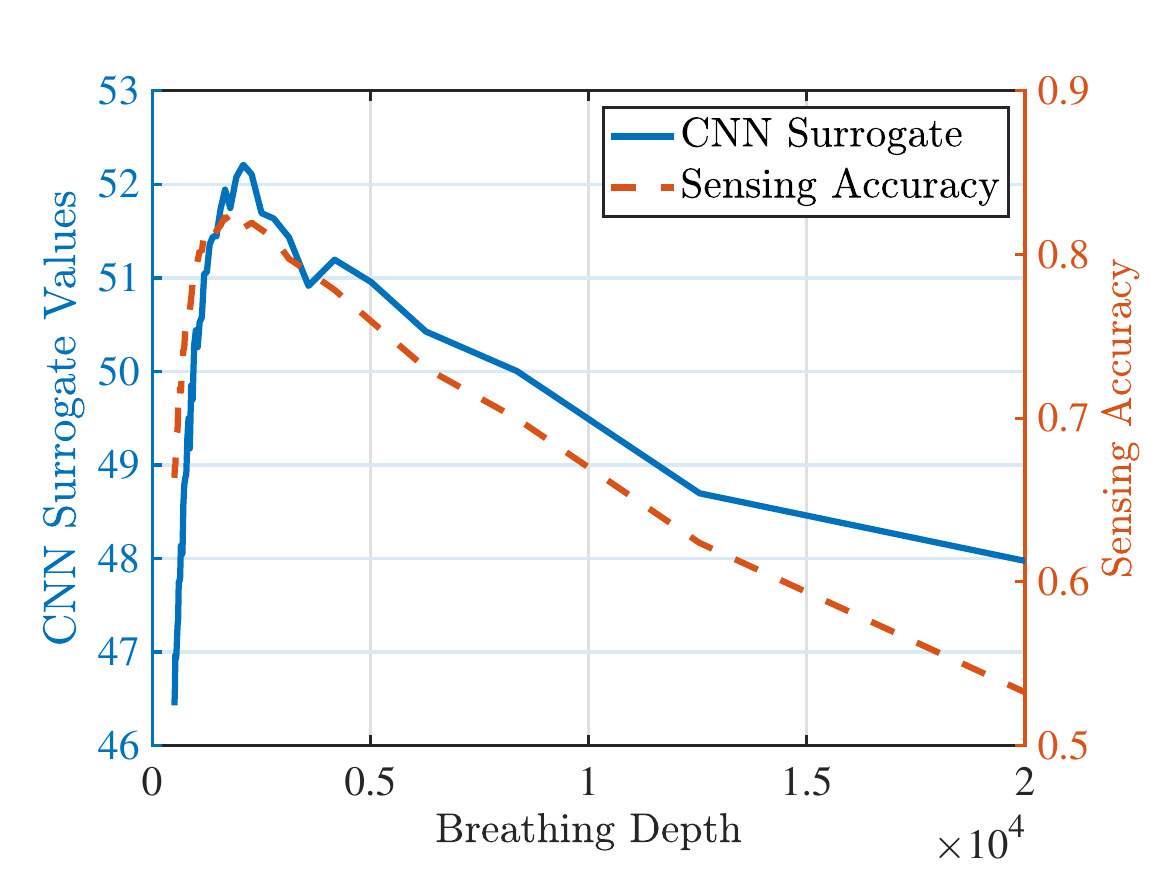}}
\caption{Spread-compression tradeoff under the constrained bandwidth.}
\label{fig:LC_DG_approx_func}
\vspace{-5mm} 
\end{figure}

\subsubsection{Benchmarking Schemes}
The performance of AirBreath sensing is evaluated through comparisons with three benchmark schemes, as described below.
\begin{itemize}
    \item \textbf{Brute-force:}
This scheme optimizes the breathing depth via a brute-force search to guarantee the optimality of the solution. Specifically, in round $n$, given the set of active sensors $\mathcal{K}_n$, the sensing accuracy on the test dataset is maximized by exhaustively searching over the feasible solution set $\{(S_n, G_n) \mid S_n G_n \leq D\}$.
The exhaustive search incurs a time complexity of $\mathcal{O}(D^2)$, rendering it impractical for realtime inference.
\item \textbf{Without Spectrum Breathing (No AirBreathing):} 
Without spectrum breathing, feature uploading is exposed to strong interference. 
In each round, the breathing depth is fixed at $S_n=D, G_n=1$, such that the full feature vector is transmitted to the edge server without compression. This is equivalent to setting the compression matrix as $\mathbf{P}=\mathbf{I}_D$.
\item \textbf{Fixed Breathing Depth (Fixed BD, $S_n=2, \frac{D}{2}$):} 
This scheme employs a fixed breathing depth across all inference rounds, without adapting to time-varying channels. i.e., $S_n=2, \frac{D}{2}, n = 1, \dots, N$.

\item \textbf{Random AirBreathing:}
Unlike the proposed AirBreath sensing with importance-driven feature compression, this scheme randomly prunes feature elements to allocate bandwidth for subsequent spread spectrum processing. 
The feature vector is compressed using a randomly generated selection matrix at each round, denoted by $\mathbf{P}\in \mathbb{R}^{D\times S}$, subject to the constraints $P_{i,j}\in{0,1}$, $\sum_i P_{i,j}\leq 1$, and $\sum_j P_{i,j}=1$. The optimal breathing depth is determined by exhaustively searching the feasible solution set to maximize the proposed surrogate functions for both linear and CNN classification.

\end{itemize}

\subsection{Compression-spreading Tradeoff}

In Fig. \ref{fig:LC_DG_approx_func}, 
we demonstrate the compression-spreading tradeoff controlled by the breathing depth for both linear and CNN classification, under the constrained system bandwidth.
The sensing performance is measured by two metrics, i.e., sensing accuracy and receive DG (CNN surrogate values).
It is observed that either the sensing accuracy or the surrogate value increases as the breathing depth increases until the maximum point.
This comes from a larger number of feature dimensions being received for inference.
After the maximum point, the sensing accuracy decreases for a continuously increasing breathing depth.
This is because the rise in breathing depth triggers a smaller processing gain and stronger interference, which reduces receive DG.
Additionally, the proposed surrogate function is seen to accurately mimic the monotonicity of sensing accuracy, especially the uniqueness of the optimal solution, which validates the Proposition \ref{prop:OPT_feature}.

\vspace{-3mm}
\subsection{Performance of AirBreath Sensing}

\begin{figure}[t!]
\centering
\subfigure[Linear Classification]{\label{fig:LC_SIR_ACC}
\includegraphics[width=0.45\columnwidth]{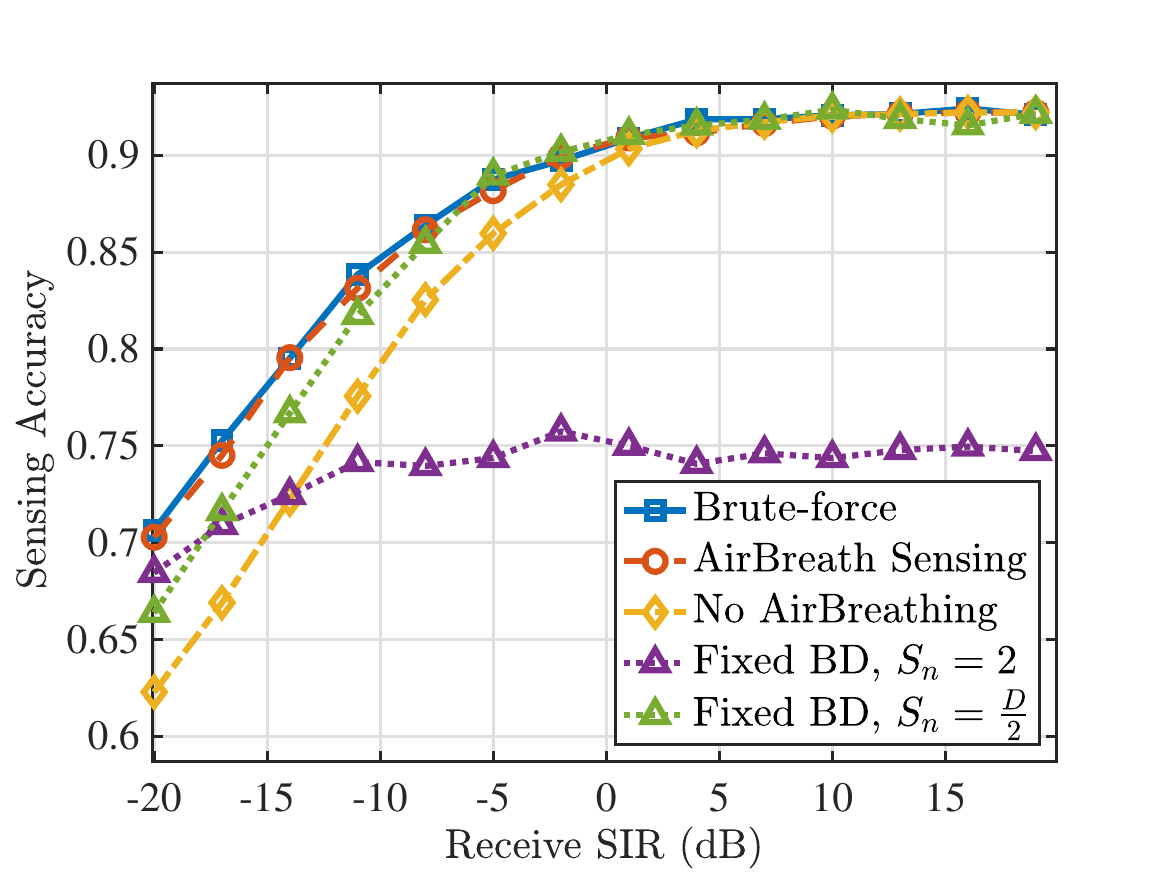}}
\subfigure[CNN Classification]{\label{CNN_SIR_ACC}
\includegraphics[width=0.45\columnwidth]{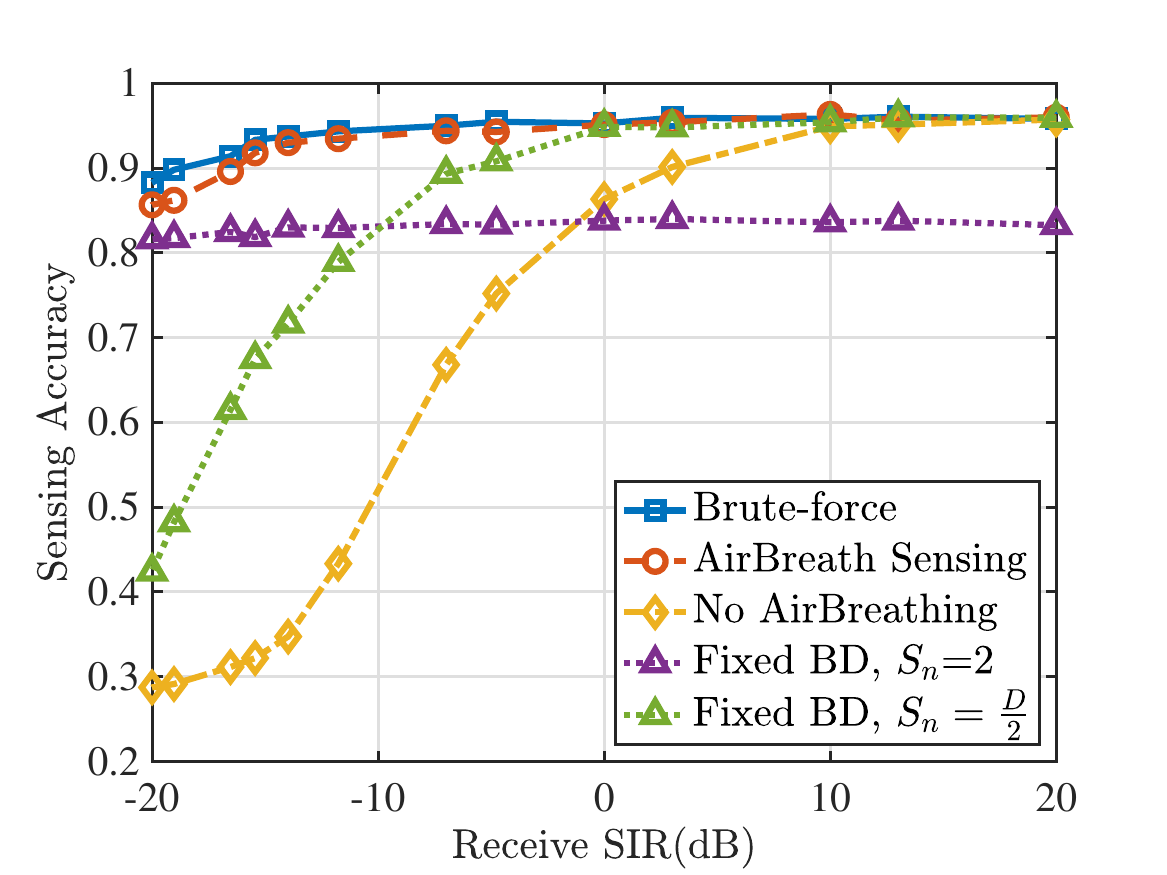}}
\caption{Performance comparison with benchmarks under different receive SIR
($K=8, P_{\sf act} = 0.5$ for linear classification and $K=9, P_{\sf act} = 0.8$ for CNN classification).}
\label{fig:SIR_ACC}
\vspace{-5mm} 
\end{figure}

\begin{figure}[t!]
\centering
\subfigure[Linear Classification]{
\includegraphics[width=0.45\columnwidth]{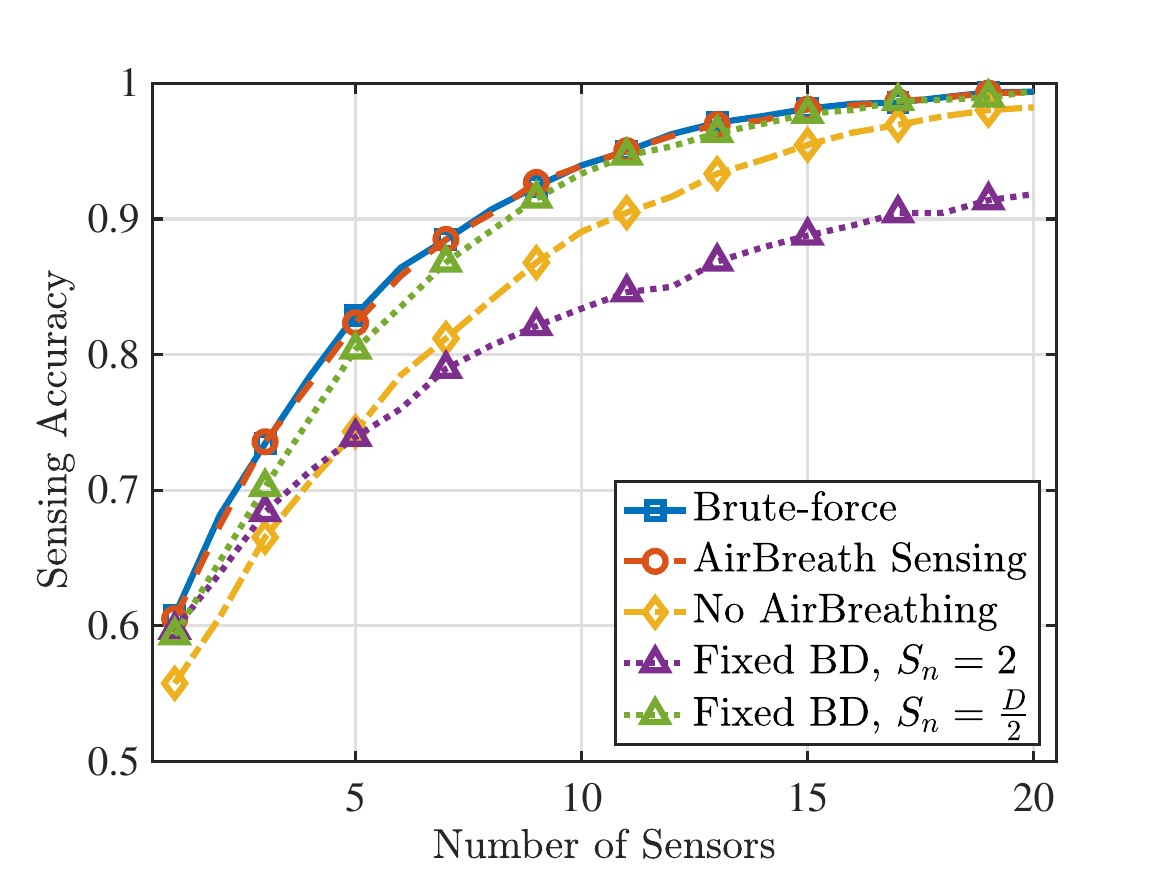}}
\subfigure[CNN Classification]{
\includegraphics[width=0.45\columnwidth]{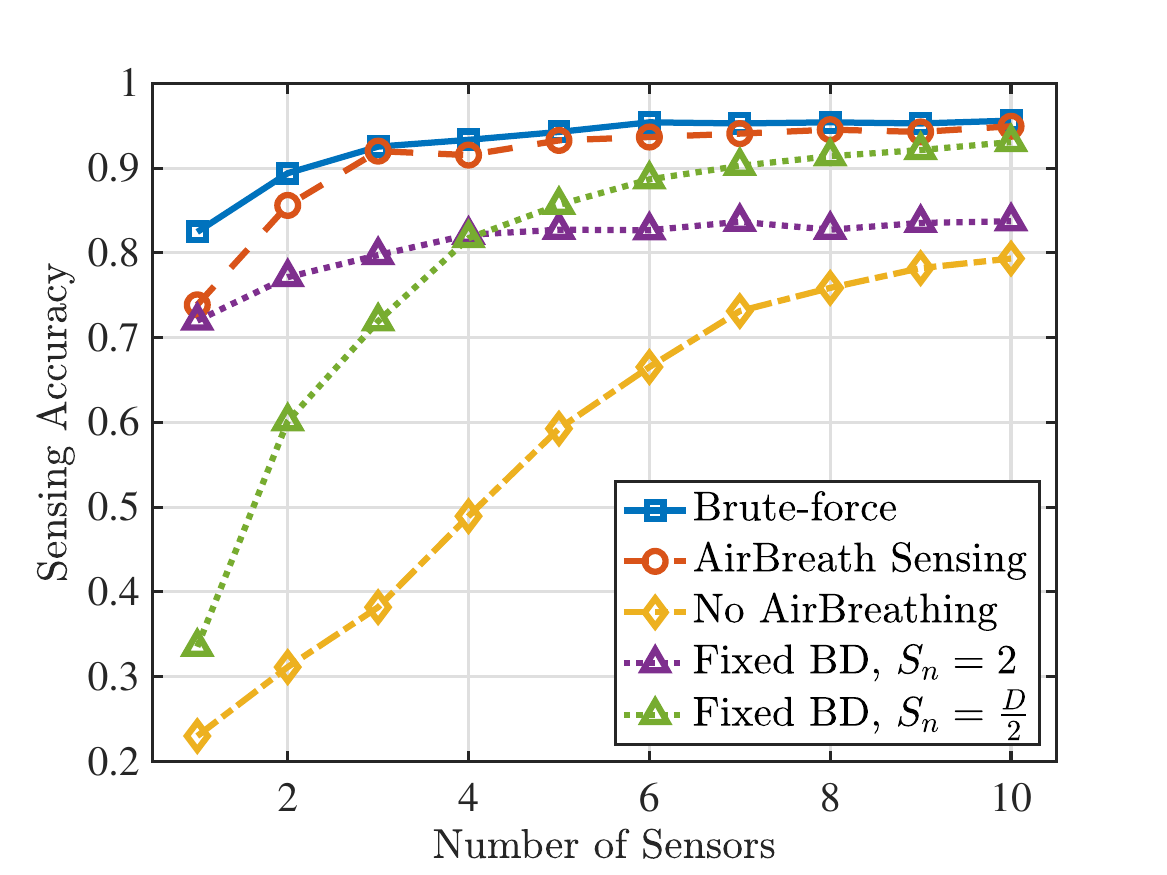}}
\caption{Performance comparison with benchmarks under different numbers of sensors
($P_{\sf act} = 0.9, \gamma_{\sf sen}=-14$ dB
 for linear classification and $P_{\sf act} = 0.9, \gamma_{\sf sen}=-4.8$ dB for CNN classification).
}
\label{fig:Sen_ACC}
\vspace{-5mm} 
\end{figure}

In this subsection, we evaluate the sensing accuracy of the proposed AirBreath sensing  by comparing it with the benchmarking schemes.
Fig. \ref{fig:SIR_ACC} demonstrates the performance comparisons with varying receive SIR for both linear and CNN classification.
The sensing accuracy increases with the received SIR and saturates for all four approaches (i.e., Brute-force, AirBreath Sensing, No AirBreathing, and Fixed BD with $S_n=\frac{D}{2}$) due to reduced channel interference. In contrast, Fixed BD with $S_n=2$ maintains a constant performance, limited by aggressive feature compression.
The proposed approach outperforms the method without spectrum breathing at low receive SIR, i.e., $\leq 5$ dB for linear classification and $\leq 10 $ dB for CNN classification, demonstrating robustness of AirBreath sensing to interference perturbation. 
This comes from the efficiency of spread spectrum in coping with interference.
The performance gap between them narrows to zero at high receive SIR.
Additionally, AirBreath sensing achieves higher sensing accuracy than the fixed breathing depth scheme in both classification cases. 
This improvement stems from its channel-adaptive breathing depth, which dynamically adjusts with the receive SIR and number of active sensors across rounds, as shown in Fig. \ref{fig:channel_optimal}.
The AirBreath sensing approaches the brute-force curve and realizes a close-to-optimal performance.
The optimality results from the accurate surrogate value in \eqref{eq:LC_surrogate} and \eqref{eq:cnn_surrog}, derived for linear and CNN classification, respectively.

\begin{figure}[t!]
\centering
\subfigure[Linear Classification]{
\includegraphics[width=0.45\columnwidth]{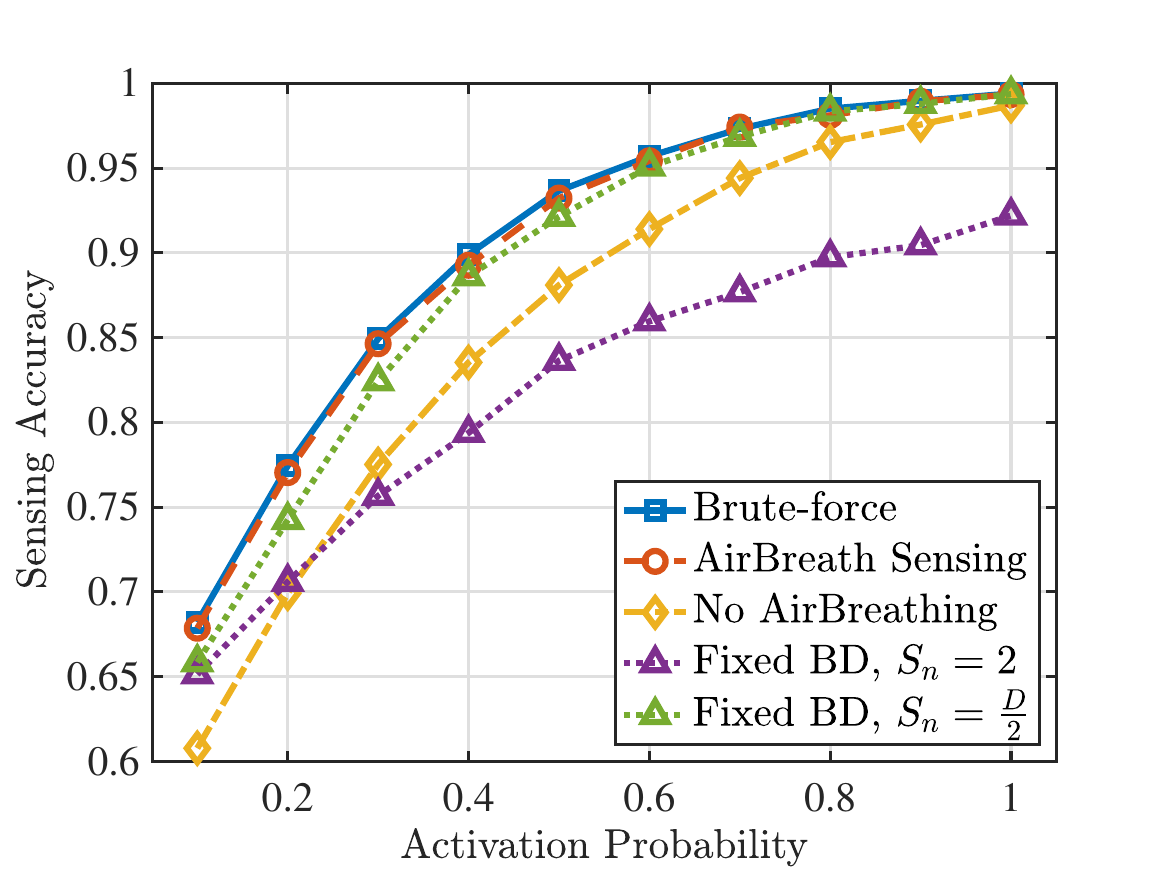}}
\subfigure[CNN Classification]{
\includegraphics[width=0.45\columnwidth]{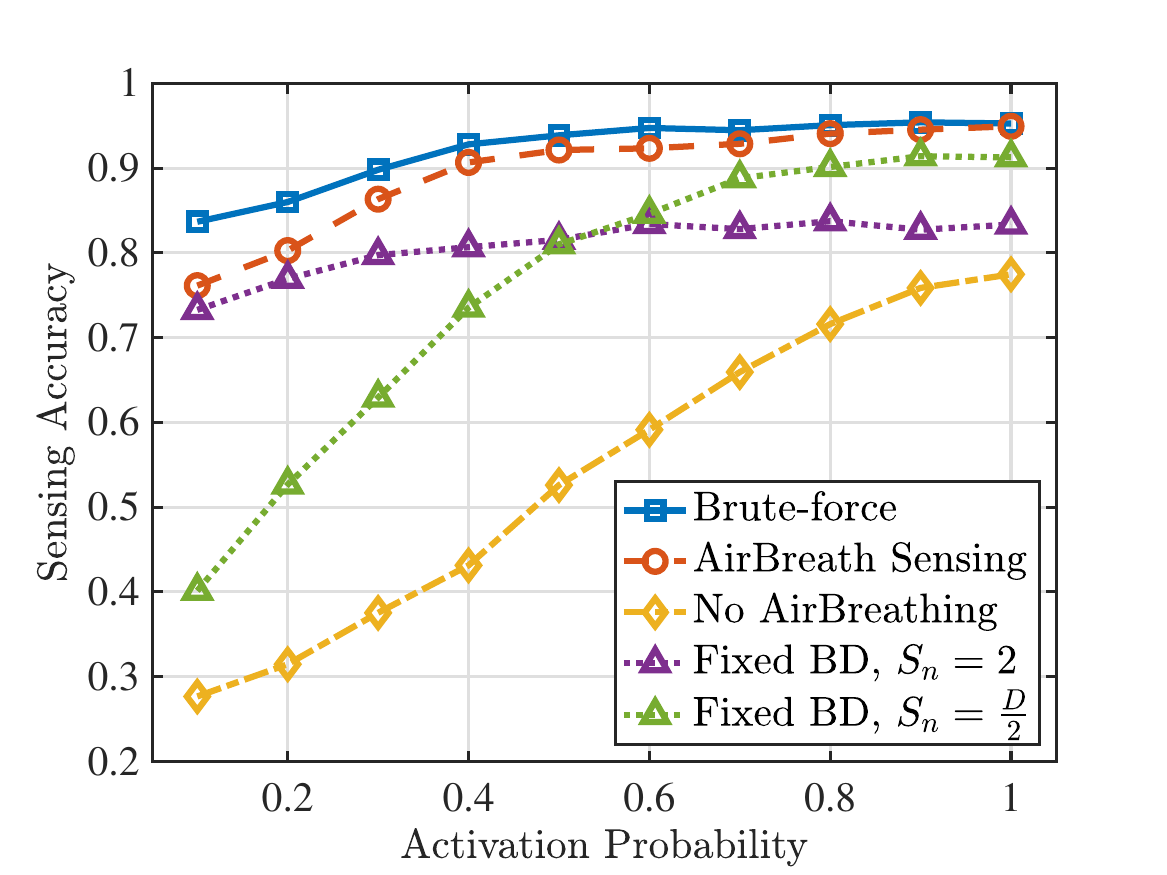}}
\caption{Performance comparison with benchmarks under different activation probabilities 
($K=18, \gamma_{\sf sen}=-14$ dB
 for linear classification and $K=8, \gamma_{\sf sen}=-4.8$ dB for CNN classification).}
\label{fig:Act_ACC}
\vspace{-3mm} 
\end{figure}

\begin{figure}[t!]
\centering
\subfigure[Linear Classification]{
\includegraphics[width=0.45\columnwidth]{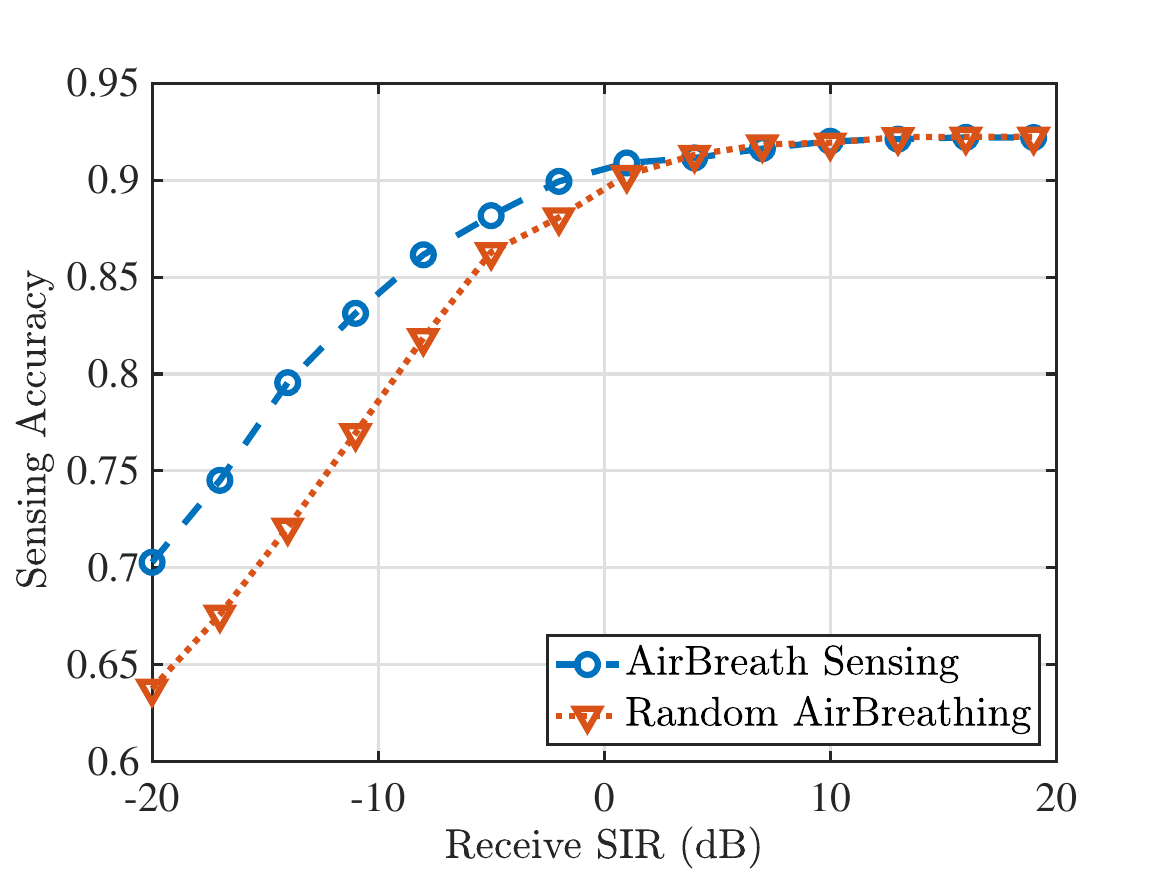}}
\subfigure[CNN Classification]{
\includegraphics[width=0.45\columnwidth]{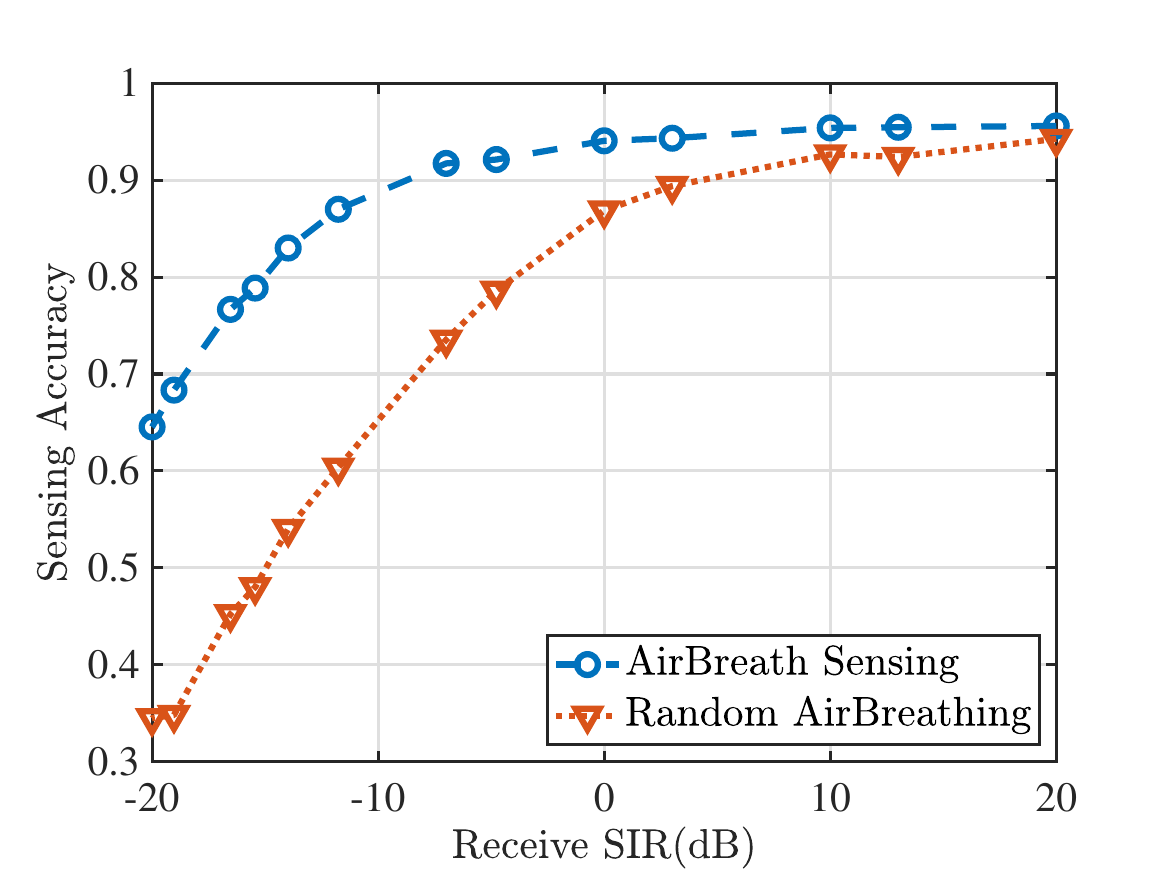}}
\caption{Performance comparison between AirBreath sensing with random feature compression
($K=8, P_{\sf act} = 0.5$).}
\label{fig:importance}
\vspace{-5mm} 
\end{figure}

In Fig. \ref{fig:Sen_ACC}, the comparisons are made over different numbers of sensors.
Increasing the sensor number enables more receive views for average pooling and an increased sensing accuracy.
In such a manner, the feature quality is enhanced by reducing the feature variance and induced channel interference in \eqref{eq:covariance_matrix}, leading to a higher DG in \eqref{eq:expicit_DG}.
The proposed approach is seen to outperform its counterpart benchmarks,  approaching the brute-force scheme as the sensor number increases.
The sensing performance is validated in Fig. \ref{fig:Act_ACC} for different settings of activation probability.
There is a monotone-increasing relation between sensing performance and activation probability. 
The reason is that higher activation probability results in more fused views for interference suppression so that the feature quality is enhanced.
In Fig. \ref{fig:importance}, AirBreath sensing outperforms its counterpart with random feature compression, especially at low receive SIR. This is because the proposed scheme prunes less important features, enabling higher compression efficiency and allowing a longer PN sequence for improved interference suppression.

\section{Conclusions}
\label{sec:conclusion}

In this work, we present a spectrum-efficient approach, called AirBreath sensing, to protect over-the-air distributed sensing from interference. 
The core design centers on controlling the level of feature compression and spread spectrum to adaptively balance the feature quantity–quality tradeoff under varying channel conditions and bandwidth constraints. Optimal breathing depth is derived for both linear and CNN classification to enhance sensing performance.
This work establishes a new design approach for robust ISEA, wherein feature sparsity is leveraged to enable efficient interference suppression via spread spectrum without incurring bandwidth expansion.

This work represents a new direction of robust over-the-air distributed sensing by integrating feature compression with interference suppression. It opens several promising avenues for future research. One direction involves multi-modal feature fusion by leveraging the orthogonal access of code division multiple access. Another is the design of AirBreath sensing under adversarial environments, including jamming detection and anti-jamming strategies.

\vspace{-2mm}

\appendix

 \begin{figure*}[t]
    \centering
    
\begin{equation}
\label{eq:B}
\begin{split}
 \mathbf{B}&=
 \begin{bmatrix}
    \mathbf{P} & \mathbf{p}_{S+1} 
   \end{bmatrix} 
   \begin{bmatrix}
       \mathbf{P}^{\sf{T}}\mathbf{C}\mathbf{P}+\beta \mathbf{I}_S &   \mathbf{P}^{\sf{T}}\mathbf{C}\mathbf{p}_{S+1}\\
      \mathbf{p}_{S+1}^{\sf{T}}\mathbf{C} \mathbf{P} & 
       \mathbf{p}_{S+1}^{\sf{T}}\mathbf{C}  \mathbf{p}_{S+1}+\beta
   \end{bmatrix}^{-1}
   \begin{bmatrix}
    \mathbf{P}^{\sf{T}} \\ \mathbf{p}_{S+1}^{\sf{T}} 
   \end{bmatrix} 
   -
   \begin{bmatrix}
    \mathbf{P} & \mathbf{p}_{S+1} 
   \end{bmatrix} 
   \begin{bmatrix}     \mathbf{P}^{\sf{T}}\mathbf{C}\mathbf{P}+\beta\mathbf{I}_S &   0\\
     0 &     0
   \end{bmatrix}^{-1}
   \begin{bmatrix}
    \mathbf{P}^{\sf{T}} \\ \mathbf{p}_{S+1}^{\sf{T}} 
   \end{bmatrix} \\
   &=  \begin{bmatrix}
    \mathbf{P} & \mathbf{p}_{S+1} 
   \end{bmatrix} 
   \underbrace{\left(
   \begin{bmatrix}
       \mathbf{P}^{\sf{T}}\mathbf{C}\mathbf{P}+\beta \mathbf{I}_S &   \mathbf{P}^{\sf{T}}\mathbf{C}\mathbf{p}_{S+1}\\
      \mathbf{p}_{S+1}^{\sf{T}}\mathbf{C} \mathbf{P} & 
       \mathbf{p}_{S+1}^{\sf{T}}\mathbf{C} \mathbf{p}_{S+1}+\beta
   \end{bmatrix}^{-1}
   -
   \begin{bmatrix}
       \mathbf{P}^{\sf{T}}\mathbf{C}\mathbf{P}+\beta\mathbf{I}_S &   0\\
     0 &     0
   \end{bmatrix}^{-1}
   \right)}_{\hat{\mathbf{B}}}
   \begin{bmatrix}
    \mathbf{P}^{\sf{T}} \\ \mathbf{p}_{S+1}^{\sf{T}} 
   \end{bmatrix},
   \end{split}
\end{equation}
   \hrulefill
\end{figure*}

\subsection{Proof of Lemma \ref{Dim_Gain}}
\label{Proof_Dim_Gain}

\begin{figure*}[t]
\centering

{\small
\begin{align}
\label{eq:B_2}
    \hat{\mathbf{B}} =&
    \begin{bmatrix}
       (\mathbf{X}^{\sf{T}} \mathbf{X}+\beta \mathbf{I}_S)^{-1} +  (\mathbf{X}^{\sf{T}} \mathbf{X}+\beta \mathbf{I}_S)^{-1}\mathbf{X}^{\sf{T}} \mathbf{y} z^{-1}   \mathbf{y}^{\sf{T}} \mathbf{X} (\mathbf{X}^{\sf{T}} \mathbf{X}+\beta \mathbf{I}_S)^{-1}&   -(\mathbf{X}^{\sf{T}} \mathbf{X}+\beta \mathbf{I}_S)^{-1}\mathbf{X}^{\sf{T}} \mathbf{y} z^{-1} \\
   - z^{-1}   \mathbf{y}^{\sf{T}} \mathbf{X} (\mathbf{X}^{\sf{T}} \mathbf{X}+\beta \mathbf{I}_S)^{-1} & 
     z^{-1}
   \end{bmatrix}
   -
   \begin{bmatrix}
       (\mathbf{X}^{\sf{T}} \mathbf{X}+\beta \mathbf{I}_S)^{-1} &   0\\
     0 &     0
   \end{bmatrix} \notag\\
     =&   \begin{bmatrix}
        (\mathbf{X}^{\sf{T}} \mathbf{X}+\beta \mathbf{I}_S)^{-1}\mathbf{X}^{\sf{T}} \mathbf{y} z^{-1}   \mathbf{y}^{\sf{T}} \mathbf{X} (\mathbf{X}^{\sf{T}} \mathbf{X}+\beta \mathbf{I}_S)^{-1}&   -(\mathbf{X}^{\sf{T}} \mathbf{X}+\beta \mathbf{I}_S)^{-1}\mathbf{X}^{\sf{T}} \mathbf{y} z^{-1} \\
   - z^{-1}   \mathbf{y}^{\sf{T}} \mathbf{X} (\mathbf{X}^{\sf{T}} \mathbf{X}+\beta \mathbf{I}_S)^{-1} & 
     z^{-1}
   \end{bmatrix}\notag\\
   =& \frac{1}{z} \underbrace{
   \begin{bmatrix}
        (\mathbf{X}^{\sf{T}} \mathbf{X}+\beta \mathbf{I}_S)^{-1}\mathbf{X}^{\sf{T}} \mathbf{y} \\ -1
   \end{bmatrix}
    \begin{bmatrix}
        ((\mathbf{X}^{\sf{T}} \mathbf{X}+\beta \mathbf{I}_S)^{-1}\mathbf{X}^{\sf{T}} \mathbf{y})^{\sf{T}} & -1
   \end{bmatrix}}_{> 0},
\end{align}
}
\hrulefill
\end{figure*}

\begin{figure*}[t]
\centering
\begin{equation}
\label{eq:block_mat_inverse}
\mathbf{A}^{-1}=\begin{bmatrix}
\mathbf{A}_1^{-1}+\mathbf{A}_1^{-1} \mathbf{A}_2\left(\mathbf{A}_4-\mathbf{A}_3 \mathbf{A}_1^{-1} \mathbf{A}_2\right)^{-1} \mathbf{A}_3 \mathbf{A}_1^{-1} & -\mathbf{A}_1^{-1} \mathbf{A}_2\left(\mathbf{A}_4-\mathbf{A}_3 \mathbf{A}_1^{-1} \mathbf{A}_2\right)^{-1} \\
-\left(\mathbf{A}_4-\mathbf{A}_3 \mathbf{A}_1^{-1} \mathbf{A}_2\right)^{-1} \mathbf{A}_3 \mathbf{A}_1^{-1} & \left(\mathbf{A}_4-\mathbf{A}_3 \mathbf{A}_1^{-1} \mathbf{A}_2\right)^{-1}
\end{bmatrix}
\end{equation}
\hrulefill
\end{figure*}


To show the monotone increase of receive DG over subspace dimensionality, we prove that $\forall \boldsymbol{\mu}_{\ell\ell'} \in \mathbb{R}^D $ and $\forall \mathbf{C} \in \mathbb{R}^{D\times D}$, the incremental DG over dimensionality is always positive.
Given the fixed processing gain $G$,
the incremental DG by increasing one feature dimension is 
\begin{equation}
    \mathcal{G}(S\!+\!1,G)\!-\! \mathcal{G}(S,G)
     \! =\!\boldsymbol{\mu}_{\ell\ell'}^{\sf{T}}\!
    \underbrace{ [\hat{\mathbf{P}}   \hat{\mathbf{C}}_{S+1}^{-1}\hat{\mathbf{P}}^{\sf{T}}\!\! - \!\mathbf{P}   \hat{\mathbf{C}}_S^{-1}\mathbf{P}^{\sf{T}} ]}_  {{\mathbf{B}}}\! \boldsymbol{\mu}_{\ell\ell'}\!,
\end{equation}
where $ \mathbf{P} = [\mathbf{p}_1, \mathbf{p}_2, \dots, \mathbf{p}_S] $ and  $ \hat{\mathbf{P}} = [\mathbf{P}, \mathbf{p}_{S+1}]$ are the projection matrices with dimensionalities of $S$ and $S+1$, respectively. $ \mathbf{p}_s \in \mathbb{R}^D, \forall s \in \{1,\dots, S+1\}$ is the $s$-th orthogonal basis of the subspace. 
This is equivalent to showing that $\mathbf{B}$ is a \emph{positive define} (PD) matrix given in \eqref{eq:B} (see the multi-line equation on the next page), where $\beta=\frac{\sigma^2_{\sf nor}}{|\mathcal{K}|^2G\gamma_{\sf sen}}$ is the scaling factor quantifying effects of the processing gain $G$ on induced interference.

To simplify the derivation, we define one matrix $\mathbf{X}=\mathbf{C}^{\frac{1}{2}}\mathbf{P}$ and vector $\mathbf{y}=\mathbf{C}^{\frac{1}{2}}\mathbf{p}_{S+1}$. Then $\hat{\mathbf{B}}$ in \eqref{eq:B} is simplified as
\begin{equation}
\small
\label{eq:B_1}
    \hat{{\mathbf{B}}} = \begin{bmatrix}
       \mathbf{X}^{\sf{T}} \mathbf{X}+\beta \mathbf{I}_S &    \mathbf{X}^{\sf{T}} \mathbf{y}\\
      \mathbf{y}^{\sf{T}} \mathbf{X} & 
       \mathbf{y}^{\sf{T}} \mathbf{y}+1
   \end{bmatrix}^{-1}
  \! -\!
   \begin{bmatrix}
       (\mathbf{X}^{\sf{T}} \mathbf{X}+\beta\mathbf{I}_S)^{-1} &   0\\
     0 &     0
   \end{bmatrix}.
\end{equation}
Note that the left term of \eqref{eq:B_1}  is the inverse of a $2\times 2$ block matrix.
Leveraging the block-wise  inverse of a matrix provided in  Lemma \ref{Lemma:mat-block-inverse}, $\hat{\mathbf{B}}$ can be decomposed as \eqref{eq:B_2},
where $z$ is a scalar given as
\begin{equation}
\label{eq:z_expression}
    z= \mathbf{y}^{\sf{T}} \mathbf{y}+1-  \mathbf{y}^{\sf{T}} \mathbf{X} (\mathbf{X}^{\sf{T}}\mathbf{X}+\beta \mathbf{I}_S)^{-1}\mathbf{X}^{\sf{T}} \mathbf{y}.
\end{equation}

\begin{Lemma}[Inverse on a $2\times 2$ block matrix~\cite{bhatia2013matrix}]
\label{Lemma:mat-block-inverse}
Let  $\mathbf{A}$ denote a $2\times 2$ block matrix:
\begin{equation}
\mathbf{A}=\begin{bmatrix}
\mathbf{A}_1 & \mathbf{A}_2 \\
\mathbf{A}_3 & \mathbf{A}_4
\end{bmatrix}.
\end{equation}
If $\mathbf{A}_1^{-1}$ or $\mathbf{A}_4^{-1}$ exist, matrix $\mathbf{A}$ can be inverted as \eqref{eq:block_mat_inverse}.
\end{Lemma}

The next step is to prove $z>0$.
Based on the Sherman-Morrison-Woodbury formula~\cite{meyer2023matrix,cui2024ice},
i.e., 
\begin{equation}
\label{eq:SMW_formula}
    \mathbf{ (A+XRY)^{-1} = A^{-1}-A^{-1}X(R^{-1}+YA^{-1}X  )^{-1}YA^{-1}  },
\end{equation}
the matrix in the form of $\mathbf{X} (\mathbf{X}^{\sf{T}}\mathbf{X}+\beta \mathbf{I}_S)^{-1}\mathbf{X}^{\sf{T}} $ in \eqref{eq:z_expression} 
 can be rewritten as
 \begin{equation}
 \label{eq:SMW}
 \begin{split}
     &\mathbf{X} (\mathbf{X}^{\sf{T}}\mathbf{X}+\beta \mathbf{I}_S)^{-1}\mathbf{X}^{\sf{T}}
     \\
    = & \mathbf{ \beta}^{-1}\mathbf{X}\mathbf{X}^{\sf T}-\beta^{-1}\mathbf{XX}^{\sf T}(\mathbf{I}_{D}+\beta^{-1}\mathbf{XX}^{\sf T})^{-1}\beta^{-1}\mathbf{XX}^{\sf T}.
 \end{split}
 \end{equation}
 Using singular value decomposition, matrix $\mathbf{\beta^{-1}XX}^{\sf T}$ can be decomposed as
 \begin{equation}
 \label{eq:SVD}
     \mathbf{\beta^{-1}XX}^{\sf T}=\mathbf{U}\Sigma\mathbf{U}^{\sf T},
 \end{equation}
where $\mathbf{U}$ is the unitary matrix, and $\mathbf{\Sigma}$ is a low-rank diagonal matrix.

Substituting \eqref{eq:SVD} and \eqref{eq:SMW} into \eqref{eq:z_expression}, the scalar $z$ can be expressed as
\begin{align}
\label{eq:z_value}
    z & = \mathbf{y}^{\sf T}\mathbf{U}(\mathbf{I}- \Sigma+\Sigma \mathbf{U}^{\sf T} (\mathbf{I}_D  +\mathbf{U}\Sigma\mathbf{U}^{\sf T})^{-1} \mathbf{U} \Sigma   ) \mathbf{U}^{\sf T}\mathbf{y} +1 \notag \\
    & = \mathbf{y}^{\sf T}\mathbf{U}(\mathbf{I}-\Sigma+\Sigma  ( \mathbf{I}_D  +  \Sigma )^{-1}  \Sigma ) \mathbf{U}^{\sf T}\mathbf{y} +1   .
\end{align}
Let $\sigma_s$ denoted the $s$-th diagonal value of  $\Sigma$. The $s$-th diagonal value of  $(\mathbf{I}-\Sigma+\Sigma  ( \mathbf{I}_S  +  \Sigma )^{-1}  \Sigma )$ in \eqref{eq:z_value}, denoted as   $\psi_s$, can be expressed as 
\begin{equation}
\begin{split}
 \psi_s &=1-\sigma_s+\sigma_s(1+\sigma_s)^{-1}\sigma_s\\
 & =1-\sigma_s+\frac{\sigma^2_s}{1+\sigma_s}=\frac{1}{1+\sigma_s}>0.
\end{split}  
\end{equation}
This gives the $z>0$. Overall, the incremental DG can be expressed as
\begin{equation}
 \mathcal{G}(S+1,G)- \mathcal{G}(S,G)=
  \frac{1}{z}\left\Vert \hat{\mathbf{p}}^{\sf{T}} \hat{\mathbf{P}}^{\sf{T}}
  \boldsymbol{\mu}_{\ell\ell'}    \right\Vert^2_2>0,
\end{equation}
where $\hat{\mathbf{p}}$ comes from \eqref{eq:B_2}, given as
\begin{equation}
    \hat{\mathbf{p}}=\begin{bmatrix}
        (\mathbf{X}^{\sf{T}} \mathbf{X}+\beta \mathbf{I}_S)^{-1}\mathbf{X}^{\sf{T}} \mathbf{y} \\ -1
   \end{bmatrix}.
\end{equation}
This completes the proof.

\vspace{-3mm}
\subsection{Proof of Lemma \ref{Proce_Gain} }
\label{Proof_Proce_Gain}

Given the fixed projection matrix $\mathbf{P}$ with the dimensionality of $S$, the incremental discriminant by increase the processing gain from $G$ to $G+1$  can be expressed as
\begin{align}
     &\mathcal{G}(S,G)- \mathcal{G}(S,G+1)\\ 
    = &  \boldsymbol{\mu}_{\ell\ell'}^{\sf{T}}\mathbf{P}
     [ (\mathbf{P}^{\sf{T}}\mathbf{C} \mathbf{P}+\beta_1 \mathbf{I}_{S})^{-1} - (\mathbf{P}^{\sf{T}}\mathbf{C} \mathbf{P}+\beta_2 \mathbf{I}_{S})^{-1} ] \mathbf{P}^{\sf T}\boldsymbol{\mu}_{\ell\ell'}, \notag
\end{align}
where $\beta_1=\frac{\sigma^2_{\sf nor}}{|\mathcal{K}|^2(G+1)\gamma_{\sf sen}}< \beta_2=\frac{\sigma^2_{\sf nor}}{|\mathcal{K}|^2G\gamma_{\sf sen}}$.
Let $\mathbf{A}_1= \mathbf{P}^{\sf{T}}\mathbf{C} \mathbf{P}+\beta_1 \mathbf{I}_{S}$ and $\mathbf{A}_2= \mathbf{P}^{\sf{T}}\mathbf{C} \mathbf{P}+\beta_2 \mathbf{I}_{S}$ define the PD and symetric matrix. We have $\mathbf{A}_1\prec \mathbf{A}_2$, i.e., $\mathbf{A}_2- \mathbf{A}_1$ is PD matrix.
Since $\mathbf{A}_1$ and $\mathbf{A}_2$ commute, i.e., $\mathbf{A}_1\mathbf{A}_2=\mathbf{A}_2\mathbf{A}_1$, this results in the  inequality $ \mathbf{A}_1 ^{-1}   \succ \mathbf{A}_2 ^{-1}$.
This gives the positive incremental DG, and completes the proof.

\vspace{-3mm}
\subsection{Proof of Proposition \ref{Lemma:DG_DG}}
\label{Proof_DG_DG}

Given the compression matrix defined in \eqref{eq:eign_projection}, the covariance matrix  of  $\hat{\mathbf{y}}$  in \eqref{eq:covariance_matrix} can be expressed as a diagonal matirx:
\begin{equation}
\begin{split}
     \hat{\mathbf{C}}  =
\frac{1}{|\mathcal{K}|}
\mathbf{P}^{\sf{T}}\mathbf{C}\mathbf{P} + \frac{\sigma^2_{\sf nor}}{G|\mathcal{K}|^2\gamma_{\sf sen}}\mathbf{I}_S 
 = \frac{\Sigma_S}{|\mathcal{K}|}  + \frac{\hat{\sigma}^2 \mathbf{I}_S}{G|\mathcal{K}|}  ,
\end{split}
\label{cov_mat_impor}
\end{equation}
where $ \Sigma_S=\diag \{\lambda_1,\dots, \lambda_S\}$ is a diagonal maxtrix and $\hat{\sigma}^2=\frac{\sigma^2_{\sf nor} }{|\mathcal{K}|\gamma_{\sf sen}}$.
Substituting the obtained \eqref{cov_mat_impor} into \eqref{eq:subspace-DG}, the receive DG that ensures the compression matrix in the eigenspace of $\mathbf{C}$, denoted by $\hat{\mathcal{G}}(S,G)$ can be simplified  by
\begin{equation}
\begin{split} 
    \hat{\mathcal{G}}(S,G)
&= |\mathcal{K}| \sum_{d=1}^S \frac{\gamma_d^2/\lambda_d}{\hat{\sigma}^2/(G\lambda_d)+1}\\
\end{split} 
\end{equation}
where $\overline{\boldsymbol{\mu}}_{\ell \ell'}^{\sf T}=\boldsymbol{\mu}_{\ell \ell'}^{\sf T} \mathbf{U}$,  $\gamma^2_{d} =\overline{\mu}^2_{\ell\ell'}(d) $ is the $d$-th element of $\overline{\boldsymbol{\mu}}_\Delta$.
This completes the proof.

\vspace{-3mm}
\subsection{Proof of Proposition \ref{prop:OPT_feature}}
\label{proof_prop:OPT_feature}

To show  the montonoicity of $ \tilde{\phi}(S)$, we take the first derivative of $ \tilde{\phi}(S)$, given as
\begin{equation}
\label{eq:derive_1}
    \tilde{\phi}'(S) =\frac{|\mathcal{K}|\zeta(S)}{\left(\tilde{\sigma}^2 S +1\right)^2},
\end{equation}
where $\zeta(S)$ is given as $   \zeta(S) = \psi'(S)\left(\tilde{\sigma}^2 S + 1\right) - \tilde{\sigma}^2 \psi(S)$.
Since the denominator of  \eqref{eq:derive_1} is  
positive (i.e., $\left(\tilde{\sigma}^2 S + 1\right)^2>0$, the sign of $\tilde{\phi}'(S)$  is determined by the sign of $\zeta(S)$.
Then the first  derivative of $\zeta(S)$ can be computed as
$
\zeta'(S) = \psi''(S)\left(\tilde{\sigma}^2 S + 1\right).
$
Since $\psi''(S)=g'(S)\leq0$ and $\tilde{\sigma}^2 S + 1\geq 0$, we have $\zeta'(S)\leq 0$.
This implies $\zeta(S)$ is a monotone-decreasing function of $S$.


Since $\zeta(S)$ is strictly decreasing over its domain, the maximizer $x^*$ of $\tilde{\phi}(S)$ must either lie at a boundary point of $S$ or satisfy $\zeta(x^*) = 0$. 
A solution to $\zeta(S) = 0$ exists precisely when $\zeta(1)$ and $\zeta(D)$ have opposite signs, i.e., $\zeta(1)\,\zeta(D) < 0$. 
Under this condition, the optimal breathing depth $S^*$ is uniquely given by the root of $\zeta(S) = 0$. 
Since $S$ must take integer values, $S^*$ can be determined by inspecting the nearest feasible integers around $x^*$ and selecting the one that maximizes $\tilde{\phi}(S)$. 
If the sign condition $\zeta(1)\,\zeta(D) < 0$ does not hold, then no interior root exists, and the maximizer is at one of the endpoints: $S^* = \arg\max_{S \in \{1, D\}} \tilde{\phi}(S)$.
This completes the argument.

\bibliography{Ref}
\bibliographystyle{IEEEtran}

\end{document}